\documentclass[11pt]{amsart}
\usepackage[utf8]{inputenc}
\usepackage[english]{babel}
\usepackage{tikz-cd}
\usepackage{graphicx}
\usepackage{amsthm, amssymb}
\usepackage{amsfonts}
\usepackage{enumerate}
\usepackage{amsmath, mathrsfs}
\usepackage{wasysym, mathtools}
\usepackage[pdfencoding=auto, psdextra]{hyperref}
\usepackage{libertine}
\usepackage{cleveref}
\usepackage[all]{xy}
\usepackage{geometry}

\newtheorem{thm}{Theorem}[section]
\newtheorem*{thm*}{Theorem}
\newtheorem{prop}[thm]{Proposition}
\newtheorem*{prop*}{Proposition}
\newtheorem{corol}[thm]{Corollary}
\newtheorem*{corol*}{Corollary}
\newtheorem{ex}[thm]{Example}
\newtheorem{lemma}[thm]{Lemma}
\newtheorem{dfn}[thm]{Definition}
\newtheorem*{dfn*}{Definition}
\newtheorem{rmk}[thm]{Remark}
\newtheorem*{notation}{Notation}
\newtheorem{claim}[thm]{Claim}
\numberwithin{equation}{thm}

\DeclareMathOperator{\Fix}{Fix}

\DeclareMathOperator{\hilb}{Hilb}
\DeclareMathOperator{\Bl}{Bl}

\DeclareMathOperator{\Pic}{Pic}
\DeclareMathOperator{\Aut}{Aut}
\DeclareMathOperator{\NS}{NS}
\DeclareMathOperator{\Sing}{Sing}
\DeclareMathOperator{\Span}{Span}
\DeclareMathOperator{\TR}{Tr}
\DeclareMathOperator{\Mon}{Mon}
\DeclareMathOperator{\id}{id}

\DeclareMathOperator{\Sym}{Sym}

\newcommand{\parallelsum}{\mathbin{\!/\mkern-5mu/\!}}

\title[Cubic threefolds and IHS manifolds]{GIT stable cubic threefolds and certain fourfolds of $K3^{[2]}$-type}

\author{Lucas Li Bassi}
\address{Lucas Li Bassi, 	DIMA (Università di Genova), Via Dodecaneso 35, 16146 Genova, Italy}
\email{lucas.libassi@dima.unige.it}
\urladdr{http://https://sites.google.com/view/lucaslibassi/}

\date{}

\begin{document}
\begin{abstract}
We study the behaviour on some nodal hyperplanes of the isomorphism, described by Boissière--Camere--Sarti in \cite{BCS}, between the moduli space of smooth cubic threefolds and the moduli space of hyperkähler fourfolds of $K3^{[2]}$-type with a non-symplectic automorphism of order three, whose invariant lattice has rank one and is generated by a class of square 6; along those hyperplanes the automorphism degenerates by jumping to another family. We generalize their result to singular nodal cubic threefolds having one singularity of type $A_i$ for $i=2, 3, 4$  providing  birational maps between the loci of cubic threefolds where a generic element has an isolated singularity of the types $A_i$ and some moduli spaces of hyperkähler fourfolds of $K3^{[2]}$-type with non-symplectic automorphism of order three belonging to different families. In order to treat the $A_2$ case, we introduce the notion of Kähler cone sections of $K$-type generalizing the definition of $K$-general polarized hyperkähler manifolds.
\end{abstract}
\maketitle

\section{Introduction}\label{sec1}

The relation between cubic hypersurfaces and irreducible holomorphic symplectic manifolds has attracted the interest of many experts during the last decades. Various examples of this interest can be found in literature. The first one is the classical result, due to Beauville and Donagi \cite{BD}, stating that the Fano variety of lines on a cubic fourfold is deformation equivalent to the Hilbert square of a $K3$ surface. There are plenty of other examples, to cite a few \cite{hassett_2000}, \cite{LehnLehnSorgervanStraten}, \cite{LSV} and the article of Boissière--Camere--Sarti \cite{BCS} which is at the core of this article.\\
 
 In \textit{loc.cit.} the authors prove the existence of an isomorphism $\psi$ between the moduli space $\mathcal{C}^{sm}_3$ of smooth cubic threefolds and the moduli space $\mathcal{N}^{\rho, \zeta}_{\langle 6\rangle}$ of fourfolds of $K3^{[2]}$-type endowed with a special non-symplectic automorphism of order three. Moreover, they analyze the extension of the period map to singular cubics, given in \cite{ACT}, in order to give a geometric interpretation of the degenerations of the automorphism $\psi$ along either the chordal or the singular nodal hyperplanes, where the cubic threefolds either acquire a nodal singularity or they are related to the chordal cubic. In particular they find a birational morphism between the stable discriminant locus (corresponding to a generic nodal degeneration) $\Delta_3^{A_1}$ and the 9-dimensional moduli space of fourfolds  of $K3^{[2]}$-type endowed with a non-symplectic automorphism	of order three, having invariant lattice isometric to $U(3)\oplus\langle -2\rangle$. In the exceptional locus of this birational morphism there are some interesting subloci, e.g. cubic threefolds having an isolated singularity of type $A_i$ for $i=2, 3, 4$.\\
 
 The aim of this article is to provide a similar result also for the closed subloci $\Delta_3^{A_2}$, $\Delta_3^{A_3}$, $\Delta_3^{A_4}\subset \Delta_3^{A_1}$ where $\Delta_3^{A_i}$ is the closure of the set of cubic threefolds having an isolated singularity of type $A_i$ for $i=1, \dots,4$ taken in the moduli space of cubic threefolds. These threefolds are of our interest because Allcock proved in \cite[Theorem 1.1]{AllcockSing} that a singular cubic threefold is GIT stable if and only if all its singularities are of type $A_i$ for $i=2, 3, 4$. Therefore, they are in the strata at the boundary of the GIT compactification of the moduli space of smooth cubic threefolds. Different types of compactifications have been studied recently in many articles, e.g. \cite{Yokoyama}, \cite{LooijengaSwiestra}, \cite{C-MGHL_Jacobians}, \cite{C-MGHL_cohomology} and the already cited \cite{AllcockSing} and \cite{ACT}.\\

 The main results of this article can be summarized as follows. 
	\begin{thm}\label{thm:main}
		The $\Delta_3^{A_i}$ locus for $i=1, \dots, 4$ is birational to a $(10-i)$-dimensional moduli space of fourfolds of $K3^{[2]}$-type with Picard group of the generic member isometric to $R_i$ endowed with a non-symplectic automorphism of order three, having invariant lattice isometric to $T_i$. These lattices are defined in the following table. \begin{center}
			\begin{tabular}{|c|c|c|} 
                \hline
				$i$ & $T_i$ & $R_i$\\ 
				\hline
				$1$ & $U(3)\oplus\langle -2\rangle$ & $U(3)\oplus\langle -2\rangle$\\ 
				\hline
				$2$ & $U\oplus A_2(-2)\oplus\langle -2\rangle$ & $U(3)\oplus\langle -2\rangle$\\ 
				\hline
				$3$ &  $U\oplus A_2(-1)^{\oplus 2}\oplus\langle -2\rangle$ & $U\oplus A_2(-1)^{\oplus 2}\oplus\langle -2\rangle$\\ 
				\hline
				$4$ & $U\oplus E_6(-1)\oplus\langle -2\rangle$ & $U\oplus E_6(-1)\oplus\langle -2\rangle$\\ 
				\hline
			\end{tabular}
		\end{center}
	\end{thm}
 
  In order to prove our results we will proceed in the following way.
  
  In \Cref{sec:moduli} and \Cref{sec:cubiche} we will set the notation and the background needed by introducing respectively basic facts about irreducible holomorphic symplectic manifolds (that will be the main object of this paper) and nodal cubics. Then, as noted in \cite[Section 4]{BCS}, in order to understand geometrically the degenerations of the automorphism along the nodal hyperplanes, one has to consider a moduli space of fourfolds of $K3^{[2]}$-type with an automorphism having an invariant lattice which is bigger than in the smooth case. So, we will start from a generic cubic threefold $C$ in $\Delta_3^{A_i}$, then we will find a $K3$ surface $\hat{\Sigma}$ having the same period of $C$. The natural choice for $\hat{\Sigma}$ will be the one used in \cite{ACT} to define the period of a nodal cubic. Indeed, as explained in \Cref{sec:cubiche} a cubic threefold having one isolated singularity can be seen as the zero locus of an homogeneous polynomial of the form
  \begin{equation*}
		f(x_0:x_1:x_2:x_3:x_4)=x_0f_2(x_1,x_2,x_3,x_4)+f_3(x_1,x_2,x_3,x_4)=0
	\end{equation*} where the $f_i$ are homogeneous polynomials of degree $i$. Then, the surface $\Sigma$ given by the following intersection: \begin{equation*}
		\begin{cases}
			f_2(x_1,x_2,x_3,x_4)=0 \\
			f_3(x_1,x_2,x_3,x_4)+x_5^3=0.
		\end{cases}
	\end{equation*} is the $K3$ surface with isolated singularities which we need to resolve. \newline
    
    After this, we will find some conditions on the Hilbert square $\hat{\Sigma}^{[2]}$ of $\hat{\Sigma}$ such that $\hat{\Sigma}^{[2]}$ is a ``good candidate'' for the relation we are looking for. Indeed, our aim is to define a birational map between $\Delta_3^{A_i}$ and some moduli space of fourfolds of $K3^{[2]}$-type. In order to do so we will define an isomorphism between generic elements via the period map. Therefore a  ``good candidate'' should be endowed with a marking, a non-symplectic automorphism and generic in a particular moduli space. Moreover, we want that the restriction of the period map to this moduli space is an isomorphism onto its image. Indeed, in general the period map for IHS manifolds is not an isomorphism because there may exist birational, non-isomorphic models in the fiber over a period. In order to ensure that this does not happen and avoid problems of separability we will use, in a first instance, the notion of $K(T)$-generality for a fourfold of $K3^{[2]}$-type, introduced by Camere in \cite[Definition 3.10]{Camere_lattice_polarized_rmks} recalled in \Cref{def:K-gen}.\newline

	In \Cref{DegSec} we will find a sufficient condition on the Picard group of $\hat{\Sigma}^{[2]}$ to make it a ``good candidate'', then we will use this result in \Cref{sec:A1} and \Cref{sec:A3A4} to prove \Cref{thm:main} for $i=1$, $3$ and $4$. For these cases the result is proven, respectively, in \Cref{prop:A1}, \Cref{prop:A3} and \Cref{prop:A4}. Note that for $i=1$ this coincides with the result stated in \cite[Proposition 4.6]{BCS}.\newline 
    
    The case $i=2$, where we start from a generic cubic in $\Delta_3^{A_2}$, behaves very differently respect to the other cases. This fact is somehow surprising at a first glance but studying it we will find out many differences with the other cases; for example, in order to find $\hat{\Sigma}$ in this case we need to blow-up three points which are permuted by a non-symplectic automorphism of order 3 instead of one fixed point as for $i=3,4$. Moreover, in the case with $i=2$ we do not have the $K(T)$-generality property, thus, there are multiple birational non-isomorphic models. Therefore, in \Cref{sec: A2} we will discover that its associated $\hat{\Sigma}^{[2]}$ does not satisfy the sufficient condition in order to be a ``good candidate'' found in \Cref{DegSec}. So, in order to deal with the $\Delta_3^{A_2}$ locus we will need to introduce in \Cref{sec:conesections} the notion of Kähler cone sections of $K$-type. This is a technical condition which generalizes the notion of $K(T)$-generality and requires some work to do. It is particularly useful when one considers a subfamily polarized by a lattice bigger than the invariant lattice in the moduli space of IHS manifolds of a given type admitting a non-symplectic automorphism with a given action in cohomology. Finally, we will show in \Cref{sec:A2works} that the notion just introduced leads us to the proof of \Cref{thm:main} for $i=2$. 

 \section{Irreducible Holomorphic Symplectic Manifols}\label{sec:moduli}
	In this section we recall briefly a few facts about irreducible holomorphic symplectic manifolds. For any integer lattice $L$ we will denote by $L_{\mathbb{C}}:=L\otimes\mathbb{C}$ its $\mathbb{C}$-linear extension. \newline
	
	\begin{dfn}
		An irreducible holomorphic symplectic (from now on IHS) manifold $X$ is a compact complex Kähler manifold which is simply connected and such that there exists an everywhere non-degenerate holomorphic 2-form $\omega_X$ such that $H^0(X, \Omega_X^2)=\mathbb{C}\omega_X$.
	\end{dfn}
 The theory of IHS manifolds has been deeply studied over the last forty years. The first examples of IHS manifolds in higher dimensions were provided by Fujiki in dimension 4, and then extended to all even dimensions by Beauville \cite{Beauville1983}. Our interest will be predominantly in one family of deformations constructed as follows. Let $\Sigma$ be a $K3$ surface; for our purposes, we can restrict to the case where $\Sigma$ is projective.
\begin{dfn}
    We define the \emph{Hilbert scheme of $n$ points on $\Sigma$} as the variety that parameterizes zero-dimensional subschemes $(Z, \mathcal{O}_Z)$ of length $n$ (i.e., $\dim \mathcal{O}_Z = n$) on the surface $\Sigma$, denoted by $\Sigma^{[n]}$. Oftentimes we will denote it also by $\hilb^n(\Sigma)$. We will refer to the Hilbert scheme $\Sigma^{[2]}$ of two points on $\Sigma$ also as the \emph{Hilbert square of $\Sigma$}.
\end{dfn}
This definition sometimes can bring, from our point of view, to a lack of geometric meaning, in the sense that it is not so clear how to generalize facts about $K3$ surfaces to Hilbert schemes of points on them. In order to deal with this fact we can give another equivalent definition.

In line with \cite[Section 6]{Beauville1983}, we use the following notations:

\begin{itemize}
\item $\Sigma^{(n)}$ stands for the variety of 0-cycles of degree $n$, defined as the quotient of $\Sigma^n:=\overbrace{\Sigma\times\dots\times\Sigma}^{n\rm\ times}$ by the symmetric group on $n$ elements. We will refer to it also with $\Sym^n(\Sigma)$.
\item We label the natural mapping associating each finite scheme with the corresponding 0-cycle (termed the \emph{Hilbert-Chow morphism}) as $\epsilon: \Sigma^{[n]}\to  \Sigma^{(n)}$.
\item We denote the locus of cycles in the form $p_1+\ldots +p_n$ such that there exists $i\neq j$ with $p_i=p_j$, also called \emph{diagonal}, as $D\subset \Sigma^{(n)}$.
\end{itemize}

\begin{dfn}[Alternative definition]
     Consider on $\Sigma^n$ the action $\gamma$ of the symmetric group $S_n$ on $\mathbb{Z}/n\mathbb{Z}$ which permutes the factors. Then the quotient $\Sigma^{(n)}=(\Sigma^n)/S_n$ is singular on the diagonal. The blow-up of $\Sigma^{(n)}$ along the diagonal is the Hilbert scheme $\Sigma^{[n]}$ of $n$ points on $\Sigma$ and the blow-up morphism is identified with the \emph{Hilbert--Chow} morphism.
\end{dfn} 
The other deformation types known at the state of art are those of the generalized Kummer manifolds, also due to Beauville \cite{Beauville1983}, plus other two examples constructed by O'Grady in dimension 6 \cite{OGrady6} and 10 \cite{OGrady10}.

 The second cohomology group $H^2(X,\mathbb{Z})$ of an IHS manifold $X$ is torsion-free and it is equipped with a nondegenerate symmetric bilinear form (known as Beauville--Bogomolov--Fujiki form), which gives it the structure of an integral lattice (see \cite{Fujiki19870}). This lattice is a deformation invariant and in literature there can be found an explicit description for each example (see e.g. \cite{debarreHKbibbia}), therefore we will say that an IHS manifold $X$ is of type $L$ if $H^2(X,\mathbb{Z})\simeq L$ implying that the deformation type is fixed.
	\begin{rmk}
		There exists no general proof of the fact that if $H^2(X,\mathbb{Z})\simeq H^2(Y,\mathbb{Z})\simeq L$, for $X$, $Y$ IHS manifolds and $L$ a lattice, then $X\sim_{\text{def}}Y$. Nevertheless, this fact is true for the known deformation families.
	\end{rmk}

 Consider an IHS manifold $X$ of type $L$. 

\begin{dfn}
A \emph{marking} on $X$ is an isometry $\eta:H^2(X,\mathbb{Z})\to L$.
A \emph{marked IHS manifold} is a pair $(X,\eta)$, where $X$ is an IHS manifold  together with a marking $\eta:H^2(X,\mathbb{Z})\to L$ on $X$.
Two marked IHS manifolds $(X_1,\eta_1)$ and $(X_2,\eta_2)$ are isomorphic if there exists an isomorphism $f:X_1\to X_2$ such that $\eta_2=\eta_1\circ f^*$.
\end{dfn}
We define the coarse moduli space of marked IHS manifolds of type $L$ as the set of marked manifolds $(X,\eta)$, modulo the equivalence relation given by isomorphism of marked manifold. We denote it with $\mathcal{M}_{L}$.

 On this space one can define a \emph{period map} 
	\begin{align*}
		\mathcal{M}_L & \rightarrow \Omega_L \\
		(X, \eta) & \mapsto \eta(H^{2,0}(X))
	\end{align*}
	where $\Omega_L:=\left\{\omega\in\mathbb{P}(L_{\mathbb{C}})\mid (\omega,\omega)=0, \ (\omega, \bar{\omega})>0\right\}$ is called \emph{period domain}. This map is a local homeomorphism \cite[Theorem 5]{Beauville1983} and surjective when restricted to any connected component $\mathcal{M}_L^{\circ}\subset\mathcal{M}_L$ \cite[Theorem 8.1]{Huybrechts2003CompactHM}. Finally, in $O(H^2(X, \mathbb{Z}))$ there exists an important subgroup $\Mon^2(X)$ called \emph{monodromy group} consisting of the parallel transport operators of $H^2(X,\mathbb{Z})$ to itself (see \cite[Definition 1.1]{Markman_surveyTorelli} for the details). This group is important due to the Hodge theoretic Torelli theorem.
	\begin{thm}[Hodge theoretic Torelli Theorem \protect{\cite[Theorem~1.3]{Markman_surveyTorelli}}]\label{thm: hodge theoretic}
		Let $X$ and $Y$ be IHS manifolds of the same deformation type. Then: 
		\begin{itemize}
			\item $X$ and $Y$ are bimeromorphic, if and only if there exists a parallel transport operator $f:H^2(X,\mathbb{Z})\to H^2(Y,\mathbb{Z})$ which is an isomorphism of integral Hodge structures.
			\item Let $f:H^2(X,\mathbb{Z})\to H^2(Y,\mathbb{Z})$ be a parallel transport operator, which is an isomorphism	of integral Hodge structures. There exists an isomorphism $\tilde{f}: X\to Y$ inducing $f$ if and only if $f$ maps some Kähler class on $X$ to a Kähler class on $Y$.
		\end{itemize}
	\end{thm} 
	We define the \emph{positive cone} of an IHS manifold $X$ as the connected component $\mathcal{C}_X$ of the set $\left\{x\in H^{1,1}(X, \mathbb{R})\mid x^2>0\right\}$ containing a Kähler class and its Kähler cone $\mathcal{K}_X\subset H^{1,1}(X, \mathbb{R})$ as the cone consisting of Kähler classes.
	\begin{dfn}[\protect{\cite[Theorem~6.2]{amerikverbitsky2014}}, \protect{\cite[Proposition 1.5]{Mongardi_notes}}]
		A monodromy birationally minimal (MBM) class is a rational class $\delta\in H^{1,1}(X)\cap H^2(X,\mathbb{Q})$ of negative square such that there exists a bimeromorphic map ${f: X\dashrightarrow Y}$ and a monodromy operator $h\in \text{Mon}^2(X)$ such that the hyperplane $\delta^{\perp}\subset H^{1,1}(X)\cap H^2(X,\mathbb{R})$ contains a face of $h(f^*(\mathcal{K}_Y))$. We denote with $\Delta(X)$ the set of integral MBM classes which are also called \emph{wall divisors}.
	\end{dfn}
	These classes are important as they give a wall and chamber structure to the positive cone and moreover the following theorem holds.
	\begin{thm}[\protect{\cite[Theorem~6.2]{amerikverbitsky2014}}]\label{thm: AmerikVerbitsky}
		Given an IHS manifold $X$ then its Kähler cone $\mathcal{K}_X$ is a connected component of $\mathcal{C}_X\setminus \mathcal{H}_{\Delta}$ with $\mathcal{H}_{\Delta}$ defined as:
		\begin{equation*}
			\mathcal{H}_{\Delta}:=\bigcup_{\delta\in\Delta(X)}\delta^{\perp}\subset H^{1,1}(X)\cap H^2(X,\mathbb{R}).
		\end{equation*}
	\end{thm}
    \begin{ex}[Numerical characterization in the $K3^{[2]}$-case]
			We point out that by \cite[Theorem 22]{HTmoving} and \cite[Theorem 1.2]{Markman_primeexceptional} there exists a numerical characterization of the elements in $\Delta(X)$ in the $K3^{[2]}$-type case. An effective class $\delta$ is a wall divisor on an IHS fourfold of $K3^{[2]}$-type $X$ if and only if satisfies one of the following:
			\begin{itemize}
				\item $(\delta, \delta)=-2$
				\item $(\delta, \delta)=-10$ and it has divisibility 2, i.e. $(\delta, H^2(X,\mathbb{Z}))\in 2\mathbb{Z}$.
			\end{itemize}
			Moreover, we recall that by \cite[Proposition 1.5]{Markman_primeexceptional} an effective class $\delta$ is monodromy reflective, i.e. the reflection by $\delta$ is an integral monodromy operator, if and only if $(\delta, \delta)=-2$. 
\end{ex}
	Moreover there exists another important wall and chamber structure given by the subset $\mathcal{B}\Delta(X)\subset\Delta(X)$ of the \emph{stably prime exceptional divisors} (see \cite[Section 6]{Markman_surveyTorelli}). Here the \emph{fundamental exceptional chamber} $\mathcal{FE}_X$ is the connected component of 
	\begin{equation*}
		\mathcal{C}_X\setminus\bigcup_{\delta\in\mathcal{B}\Delta(X)}\delta^{\perp}
	\end{equation*}
	containing the Kähler cone and by \cite[Proposition 5.6]{Markman_surveyTorelli} the birational Kähler cone $\mathcal{BK}_X$ satisfies the following $\mathcal{BK}_X\subset\mathcal{FE}_X\subset\overline{\mathcal{BK}}_X$. Given a marking $\eta$ and fixing a connected component in $\mathcal{M}^{\circ}_L\subset\mathcal{M}_L$ we can translate these definitions to the lattice, e.g. $\Delta(L)$ will be the set consisting of elements $\eta(\delta)$ with $\delta$ a wall divisor for $(X,\eta)\in\mathcal{M}^{\circ}_L$ and analogously for $\mathcal{B}\Delta(L)$. Then we define $\mathcal{C}_L:=\left\{x\in L\otimes\mathbb{R}\mid (x,x)>0\right\}$ and, given again $(X,\eta)\in\mathcal{M}^{\circ}_L$, the monodromy group of $L$ is $\Mon^2(L):=\eta\circ\Mon^2(X)\circ\eta^{-1}$. \newline Now we continue our review talking about automorphisms on an IHS manifold $X$. A natural way to characterize them is to look at their action on the symplectic form $\omega_X$, in fact an automorphism $\sigma$ is said \emph{symplectic} if its action is trivial on $\omega_X$, i.e. $\sigma^*_{\mathbb{C}}(\omega_X)=\omega_X$, and non-symplectic otherwise. If no non-trivial power of a non-symplectic automorphism is symplectic then the automorphism is called \emph{purely non-symplectic}; this is always the case when the order is prime. In this last case one can check with a simple computation that if an element is in the invariant lattice $H^2(X, \mathbb{Z})^{\sigma^*}$ then it is orthogonal to the symplectic form $\omega_X$. Therefore, $H^2(X, \mathbb{Z})^{\sigma^*}$ is contained in the Néron--Severi lattice $\NS(X)$. Moreover, by \cite[Proposition 6]{Beauvilleremarks}, an IHS manifold admitting a non-symplectic automorphism is always projective. 
	\subsection{Moduli Spaces and period maps}
	In this section we review the construction of the moduli spaces of lattice polarized IHS manifolds. These notions were first given and analyzed in \cite{BCS_complex} for the $K3^{[n]}$-type deformation, then \cite{brandhorstCattaneo} generalized their result to the other deformation families. \newline
    
	Given an automorphism $\rho$ of a lattice $L$ and an embedding $j: T\hookrightarrow L$ of the invariant lattice $T\simeq L^{\rho}$ we give the following definition.
	\begin{dfn}\label{def: rho_j polarizzazione}
		A $\left(\rho, j \right)$-polarization of an IHS manifold $X$ of type $L$ consists of the following data:
		\begin{enumerate}[i)]
			\item a marking $\eta$;
			\item a primitive embedding $\iota\colon T\hookrightarrow \Pic(X)$ such that $\eta\circ\iota= j$;
			\item an automorphism $\sigma\in\Aut(X)$ such that $\sigma^*_{|H^{2,0}(X)}=\zeta\cdot\text{id}$ (with $\zeta$ a primitive n-th root of unity) and $\eta$ is a framing for $\sigma$, i.e. the following diagram commutes
			\begin{equation*}
				\begin{tikzcd}
					H^2(X,\mathbb{Z})\ar[r, "\sigma^*"] \ar[d, "\eta"] & H^2(X,\mathbb{Z}) \ar[d, "\eta"] \\
					L\ar[r, "\rho"] & L
				\end{tikzcd}
			\end{equation*}
		\end{enumerate}
	\end{dfn} 
	
	The period domain is in this case (see \cite[Section 3.2]{brandhorstCattaneo})
	\begin{equation*}
		\Omega^{\rho,\zeta}_{T}:=\left\{x\in\mathbb{P}(S_\zeta)\mid h_{S}(x,x)>0\right\},
	\end{equation*}
	where we denoted with $S$ the orthogonal complement of $T$ in $L$ and with $S_\zeta$ the eigenspace relative to $\zeta$ inside $S_{\mathbb{C}}$. Given a sublattice $S\subset L$ we denote as $\Delta(S):=\Delta(L)\cap S$. Then, the period map on a connected component of the moduli space of $\left(\rho, j \right)$-polarized IHS manifolds $X$ of type $L$ is surjective on
	\begin{equation*}
		\Omega^{\rho,\zeta}_{T}\setminus\bigcup_{\delta\in\Delta(S)}(\delta^{\perp}\cap\Omega^{\rho,\zeta}_{T})
	\end{equation*}
	by \cite[Proposition 3.12]{brandhorstCattaneo}, but in order to give a bijective restriction we need to introduce another definition. Choose $K(T)$ as a connected component of $$\mathcal{C}_T\setminus\bigcup_{\delta\in\Delta(S)}\delta^{\perp}\subset T_\mathbb{R}.$$
	\begin{dfn}
		A $(\rho, j)$-polarized manifold $(X, \eta)$ is $K(T)$-general if the cone generated by the invariant Kähler classes $\mathcal{K}_X^{\sigma^*}$ is identified with $K(T)$ under the polarization $\eta$, i.e. $\eta(\mathcal{K}_X^{\sigma^*})=K(T)$.
	\end{dfn}
	Now define the following $$\Gamma^{\rho, \zeta}_{T}:=\left\{\gamma_{|S}\in O(S)\mid\gamma\in O(L), \, \gamma_{|T}=\text{id}, \, \gamma\circ\rho=\rho\circ\gamma\right\}$$ and $$\Delta'(L):=\left\{\nu\in\Delta(L)\mid \nu=\nu_T+\nu_S,\, \nu_T\in T_\mathbb{Q},\, \nu_S\in S_\mathbb{Q},\, \nu_T^2, \nu_S^2<0\right\}.$$ Moreover, we set $\mathcal{H}_T:=\bigcup_{\delta\in\Delta(S)}\delta^{\perp}$ and $\mathcal{H}'_{T}:=\bigcup_{\delta\in\Delta'(S)}\delta^{\perp}$ where again we use the convention that for a sublattice $S\subset L$ we denote $\Delta'(S):=\Delta'(L)\cap S$. Then \cite[Theorem 5.6, Proposition 6.2]{BCS_complex} state that the restriction of the period map to the moduli space of $K(T)$-general IHS manifolds of type $L$ $$\mathcal{M}^{\rho, \zeta}_{K(T)}\to \Omega^{\rho, \zeta}_{T}\setminus\left(\mathcal{H}_{T}\cup\mathcal{H}'_{T}\right)$$ is an isomorphism and it induces an isomorphism on the quotients: \begin{equation} \label{eq: Moduli}
		\mathcal{P}^{\rho, \zeta}_{K(T)}:\mathcal{N}^{\rho, \zeta}_{K(T)}:=\frac{\mathcal{M}^{\rho, \zeta}_{K(T)}}{\text{Mon}^2(T, \rho)}\to\frac{\Omega^{\rho, \zeta}_{T}\setminus\left(\mathcal{H}_{T}\cup\mathcal{H}'_{T}\right)}{\Gamma^{\rho, \zeta}_{T}}
	\end{equation} 
	where we denoted with $\text{Mon}^2(T, \rho)$ the group of $(\rho, T)$-polarized monodromy operators:\begin{equation*}
		\text{Mon}^2(T, \rho):=\left\{g\in\Mon^2(L)\mid g_{|T}=\id, \, g\circ\rho=\rho\circ g\right\}.
	\end{equation*}
	We want now to give a more general notion of polarization and introduce the $(M, j)$-polarization for a lattice $M$ of signature $(1,t)$ with a primitive embedding $j:M\subset L$ defined in \cite[Definition 3.1]{CAMERE_lattice_polarized} as follows.
	\begin{dfn}\label{rmk: Mpolarization}
		Given an IHS manifold $X$ of type $L$ we say that it carries an $(M, j)$-polarization if it has: \begin{enumerate}
			\item a marking $\eta\colon H^2(X, \mathbb{Z})\to L$;
			\item a primitive embedding $\iota\colon M\hookrightarrow\Pic(X)$ such that $\eta\circ\iota= j$.
		\end{enumerate} 
	\end{dfn}
	\begin{rmk}
		The $(\rho, j)$-polarization is a special type of $(M, j)$-polarization with $M=T$, the invariant lattice for the automorphism $\rho$, and on which we ask the existence of an automorphism satisfying item iii) of \Cref{def: rho_j polarizzazione}.
	\end{rmk}
	Moreover, it is not hard to see that the notion of $K(T)$-generality derives from the notion of $K(M)$-generality for an $(M,j)$-polarized IHS manifold given in \cite[Definition 3.10]{Camere_lattice_polarized_rmks}. Indeed, let $\mathcal{C}_M$ be the connected component of the positive cone such that $\iota(\mathcal{C}_M)$ contains the Kähler cone $\mathcal{K}_X$ of an $(M,j)$-polarized IHS manifold $X$. We define also $K(M)$ as a connected component (also called \emph{chamber}) of $$\mathcal{C}_M\setminus\bigcup_{\delta\in\Delta(M)}\delta^{\perp}\subset M\otimes \mathbb{R}.$$
	\begin{dfn}\label{def:K-gen}
		An $(M, j)$-polarized IHS manifold $(X, \eta)$ is $K(M)$-general if $\iota(K(M))=\mathcal{K}_X\cap\iota(\mathcal{C}_M)$.
	\end{dfn}
	In order to see that this definition is a generalization of the one given above suppose that $(X, \eta)$ is a $(T, j)$-polarized IHS manifold of type $L$ where $T$ is the invariant lattice of an automorphism $\rho\in O(L)$. Suppose moreover that the $(T, j)$-polarization extends in a natural way to a  $(\rho, j)$-polarization, i.e. condition iii) of \Cref{def: rho_j polarizzazione} is satisfied. Then, as $T$ is the invariant sublattice of $\rho$, both definitions of $\mathcal{C}_T$ coincide. For the same reason $\mathcal{K}^{\sigma^*}_X=\mathcal{K}_X\cap \iota (\mathcal{C}_T)$. Looking at \Cref{def: rho_j polarizzazione}, we see that $\mathcal{K}_X\cap \iota (\mathcal{C}_T)=\mathcal{K}^{\sigma^*}_X=\iota(K(T))$ is equivalent to requiring that $\eta(\mathcal{K}_X\cap \iota (\mathcal{C}_T))=\eta(\mathcal{K}^{\sigma^*}_X)=K(T)$. In \textit{loc. cit.} the author finds also a period map and its injective restriction with statements similar to the $(\rho, j)$-polarized case.

\section{Cubic threefolds}\label{sec:cubiche}
In this section we define the objects we are mainly interested in, i.e. the nodal cubic threefold $C$ and the $K3$ surface whose Hilbert square will be our ``good candidate'' as said in the introduction. The construction highlighted here is standard, a good reference for it is \cite{viktorova2024classificationsingularcubicthreefolds}.\newline

	Let $C\subset \mathbb{P}^4$ be a cubic threefold with an isolated singularity in $p_0\in\mathbb{P}^4$, which we may assume to be $(1:0:...:0)$. Thus its equation is of the form \begin{equation*}
		f(x_0:x_1:x_2:x_3:x_4)=x_0f_2(x_1,x_2,x_3,x_4)+f_3(x_1,x_2,x_3,x_4)=0
	\end{equation*} where the $f_i$ are homogeneous polynomials of degree $i$ in $\mathbb{C}[x_1,x_2,x_3,x_4]$ sufficiently generic in $|\mathcal{O}_{\mathbb{P}^3}(3)|$.
	Let now $Y\subset \mathbb{P}^5$ be the triple cover of $\mathbb{P}^4$ branched over $C$, which is then described by the vanishing of the following polynomial \begin{equation*}
		F(x_0:x_1:x_2:x_3:x_4:x_5)=x_0f_2(x_1,x_2,x_3,x_4)+f_3(x_1,x_2,x_3,x_4)+x_5^3
	\end{equation*}
	using the same notation as above. This hypersurface has again an isolated singularity of type ADE at the point $p\in\mathbb{P}^5$ of coordinates $(1:0:...:0)$. 
    \begin{dfn}
        We say that a cubic fourfold $Y\subset\mathbb{P}^5$ is \emph{associated} to a cubic threefold $C\subset\mathbb{P}^4$ when $Y\to \mathbb{P}^4$ is a triple cover branched on $C$.
    \end{dfn}
	Let us consider now the hyperplane $H\subset\mathbb{P}^5$ given by $\{x_0=0\}$. In this hyperplane, which we will identify with $\mathbb{P}^4$ of coordinates $(x_1:x_2:x_3:x_4:x_5)$, we consider the surface $\Sigma$ given by the following intersection: \begin{equation}\label{eq:4}
		\begin{cases}
			f_2(x_1,x_2,x_3,x_4)=0 \\
			f_3(x_1,x_2,x_3,x_4)+x_5^3=0.
		\end{cases}
	\end{equation}
	This is the complete intersection of a quadric $Q$ (defined by $f_2=0$) and a cubic $K$ (defined by $f_3+x_5^3=0$) in $\mathbb{P}^4$ when $f_2$ and $f_3$ are sufficiently generic. This surface is deeply linked to the fourfold $Y$ and a theorem by Wall \cite{Wall} links the singularities of $\Sigma$ to those of the blow-up $\Bl_p(Y)$ of $Y$ at $p$. \begin{thm}[\protect{\cite[Theorem~2.1]{Wall}}]\label{thmWall}
		Let $q$ be a singular point of $\Sigma$. If both $Q$ and $K$ have a singularity in $q$ then the whole line $\bar{pq}$ connecting p and q is singular in $Y$. \newline If $q$ is not a singularity of both $Q$ and $K$ and is an ADE singularity of type \textbf{T} for $Q$ or $K$ then one of the followings holds:
		\begin{enumerate} [i)]
			\item $Q$ is smooth at $q$ and the cubic fourfold $Y$ has exactly two singularities, namely $p$ and $p'$, on the line $\bar{pq}$ and $p'$ is of type \textbf{T}.
			\item $Q$ is singular at $q$ and the line $\bar{pq}$ meets $Y$ only in $p$ and the blow-up $\Bl_p(Y)$ of $Y$ in $p$  has a singularity of type \textbf{T} at $q$.
		\end{enumerate}
	\end{thm}
	As we asked $p$ to be the only singularity on $Y$, the only possibility is the one described by item $ii)$ of the \Cref{thmWall}. Thus the possibilities for the singularities of $\Sigma$ are exactly those that can be found in the following table, based on \cite[Lemma~2.1]{DR01}. 
	\begin{center}
		\begin{tabular}{|c||c|c|c|c|c|c|c|c|} 
			\hline
			\textbf{T} & \boldmath $A_1$ & \boldmath $A_2$ & \boldmath $A_{n\geq 3}$ & \boldmath $D_{4}$ & \boldmath $D_{n\geq 5}$ & \boldmath $E_6$ & \boldmath $E_7$ & \boldmath $E_8$\\ 
			\hline
			\boldmath$\hat{T}$ & $\emptyset$ & $\emptyset$ &  \boldmath $A_{n-2}$ & 3\boldmath $A_1$ & \boldmath $A_1$ +\boldmath $D_{n-2}$ & \boldmath $A_5$ &  \boldmath $D_{6}$ & \boldmath $E_7$\\
			\hline
		\end{tabular}
	\end{center}
	Here \boldmath $\hat{T}$ is the type of the singularities that one can find on the exceptional divisor of the blow-up of a variety in a point \unboldmath$p$ that has a singularity of type \textbf{T}. \newline
	Another interesting observation on $\Sigma$ can be done following the argument of C. Lehn (\cite[Lemma 3.3, Theorem 3.6]{lehn2015twisted}) and Hassett (\cite[Lemma 6.3.1]{hassett_2000}). 
	\begin{thm}\label{Lehn}
		Let $Y\subset\mathbb{P}^5$ be a cubic fourfold with simple isolated singularities and suppose that it is neither reducible, nor a cone over a cubic threefold. Let $p\in Y$ be a singular point and assume that there exist no planes $\Pi\in Y$ such that $p\in\Pi$. Then the minimal resolution of $\Sigma:=F(Y,p)$, the Fano variety of lines in $Y$ passing through $p$, is a $K3$ surface. Moreover, $F(Y)$, the Fano variety of lines in $Y$, is birational to $\hilb^2(\Sigma)$
	\end{thm}
	\begin{proof}
		We write here the explicit morphism as it will be useful for the next sections. 

For the first part, see the discussion above. Moreover, $\Sigma$ is a (2,3)-complete intersection in $\mathbb{P}^4$ having only isolated ADE singularities, so it admits a minimal model which is a $K3$ surface.

For the rest of the proof, consider $W\subset Y$ the cone over $\Sigma$ with vertex $p$. This is a Cartier divisor on $Y$ cut out by the equation $f_2=0$. Hence a generic line $l\subset Y$ intersects $W$ in exactly two points counted with multiplicity, thus defining a closed subscheme $\xi_{l\cap W}$ of length two on $\Sigma$. Therefore, we can define the birational map 
\begin{align*}
\varphi^{-1}:F(Y) & \dashrightarrow \hilb^2(\Sigma) \\
l & \mapsto \xi_{l\cap W}
\end{align*}
The birational inverse of $\varphi$ is given by the natural map
\begin{align*}
\varphi:\hilb^2(\Sigma) & \to F(Y) \\
\xi & \mapsto l_{\xi}
\end{align*}
where we define the residual line $l_{\xi}$ as follows. The intersection between $Y$ and $\langle\xi, p\rangle\simeq\mathbb{P}^2$ consists of a cone over $\xi$ and a line $l_\xi$. 
	\end{proof}
 \begin{rmk}\label{rmk: indeterminacy}
		Note that $\varphi$ has no indeterminacy points. Moreover, note that the indeterminacy locus of $\varphi^{-1}$ is contained in $F(Y, p)\simeq \Sigma$. Indeed, looking at the definition of $\varphi^{-1}$  in the proof above we can see that it is not defined when a line $l\subset Y$ is contained in $W$, the cone over $\Sigma$ with vertex $p$. This means that either $l\subset \Sigma$ or $p\in l$, the former is impossible otherwise the plane $\Pi_{l,p}:=\langle l, p\rangle$ would be contained in $Y$.
	\end{rmk}
 \begin{rmk}\label{rmk: planes}
     The condition of not having planes passing through the singular point of a cyclic cubic fourfold is a generic condition. For this reason from now on we will suppose that this condition is satisfied by every cyclic cubic fourfold appearing also when not explicitly said.
 \end{rmk}
 \subsection{Moduli space of cubic threefold as a ball quotient}\label{sec: riassunto ACT}
 In this section we recall Allcock--Carlson--Toledo’s construction of a period map for the moduli space of GIT stable cubic threefolds as done in \cite{ACT}. This section is not to be intended as a complete overview of their work but as a recollection of their results useful to understand the following sections.\newline

 We denote by $\mathcal{C}_3^s:=|\mathcal{O}_{\mathbb{P}^4}(3)|\parallelsum PGL_5(\mathbb{C})$ the GIT moduli space of $PGL_5(\mathbb{C})$-stable of stable cubic threefolds and $\mathcal{C}_3^{sm}$ the sublocus of smooth cubic threefolds (this is the open set determined by the nonvanishing of the discriminant as shown in \cite[Chapter 5]{mukai_2003}). Now, given $C\in\mathcal{C}_3^{sm}$ the idea outlined by Allcock--Carlson--Toledo is to use the associated cyclic cubic fourfold to induce a period map on $C\in\mathcal{C}_3^{sm}$. Indeed, for any cubic fourfold $Y$ we can define a \emph{marking}, i.e. an isometry $$\eta:H^4_{\circ}(Y, \mathbb{Z})\to S(-1)\simeq U^{\oplus 2}\oplus E_8^{\oplus 2}\oplus A_2$$ of the middle primitive cohomology. Moreover, the \emph{period} of the marked pair $(Y, \eta)$ is just $[\eta(H^{3,1}(Y))]\in\mathbb{P}(S(-1)\otimes \mathbb{C})$. Let $\sigma$ be the covering automorphism of the associated cubic fourfold $Y$. Given a marking $\eta$ for $Y$ we can define the abstract isometry induced by $\sigma$ as $\rho:=\eta\circ\sigma\circ\eta^{-1}$ and we define a \emph{framing} as the equivalence class of markings $\tilde{\eta}$ compatible with $\rho$, i.e. $\tilde{\eta}\circ\rho=\rho\circ\tilde{\eta}$, up to action of $\mu_6:=\{\pm \id_{S(-1)}, \pm \rho, \pm \rho^2\}$.

 Now, we denote by $\mathcal{F}^{sm}_3$ the moduli space of framed smooth cubic threefolds and by $\Gamma:=\{\gamma\in O(S(-1))\mid \gamma\circ\rho=\rho\circ\gamma\}$. The latter acts on the former by composition with the framing, i.e. $(C,\eta)\mapsto(C,\gamma\circ\eta)$. As $\mu_6\subset\Gamma$ acts trivially on $\mathcal{F}^{sm}_3$ we consider $\mathbb{P}\Gamma:=\Gamma/\mu_6$ and $\mathcal{C}_3^{sm}\simeq\mathcal{F}^{sm}_3/\mathbb{P}\Gamma$. So, any framing $\eta: H^4_{\circ}(Y, \mathbb{Z})\to S(-1)$ induces an isomorphism $\eta: H^4_{\circ}(Y, \mathbb{Z})\to S(-1)_\zeta$, where $S(-1)_\zeta$ is the eigenspace of $S(-1)\otimes\mathbb{C}$ for the eigenvalue $\zeta$ of the isometry $\rho$. Here by $\zeta$ we denote a primitive third root of unity. Note that if we act on a a marked cubic fourfold of period $\eta(H^{3,1}(Y))$ with an element of $\mu_6$ we get that the period is multiplied by a non-zero scalar, therefore it remains well defined on the framed cubic threefolds. Therefore we have the following.

\begin{thm}
    The period map sending a framed cubic threefold $(C,\eta)$ to $[\eta(H^{3,1}(Y))]\in\mathbb{P}(S(-1)_\zeta)$ is an isomorphism onto the image equivariant with respect to the action of $\mathbb{P}\Gamma$. Moreover, the image is the complement of an hyperplane arrangement $\mathcal{H}$ with:\begin{equation*}
        \mathcal{H}:=\bigcup_{\delta\in S(-1), \, \delta^2=2}\delta^{\perp}.
    \end{equation*}
\end{thm}
\begin{proof}
    See \cite[Theorem 1.9]{ACT}.
\end{proof}
In \textit{loc. cit.} the authors study also an extension of the period map for the GIT stable cubic threefolds (see \cite{AllcockSing} for the details on GIT stability of cubic threefolds). In particular in \cite[Chapter 6]{ACT} the authors show, by studying the limit Hodge structure of the nodal degeneration of a cubic threefold, that the period map can be extended to $\Delta_3^{A_1}$ using the period of its associated $K3$ surface, i.e. the one defined in \Cref{sec:cubiche}. 
\begin{thm}[\protect{\cite[Theorem 6.1]{ACT}}]\label{EstensioneACT}
    The period map above defined can be holomorphically extended to an isomorphism between the GIT stable locus $\mathcal{C}^s_3$ and its image, mapping $\Delta_3^{A_1}$ to a divisor.
\end{thm}

	\subsection{Motivating example}\label{ex:fondante}
		
  Here we introduce the example which will be the core of our analysis. Given, as before, a ramified cyclic covering $Y\to\mathbb{P}^4$ branched along the cubic $C$ there exists a covering automorphism $\sigma$ on $Y$ acting by multiplication by a primitive third root of unity $\zeta$.  Any marking of the middle primitive cohomology $H^4_{\circ}(Y, \mathbb{Z})\to S(-1)$ can be composed with the Abel--Jacobi map in order to induce a marking on the middle primitive cohomology of the Fano variety of lines $F(Y)$.
		\begin{equation*}
			\begin{tikzcd}
				H^4_{\circ}(Y, \mathbb{Z}) \ar[r, "\tilde{\eta}"] \ar[d, "A"] & S(-1) \ar[d, "-\text{id}"] \\
				H^2_{\circ}(F(Y), \mathbb{Z}) \ar[r, "\eta"] & S.
			\end{tikzcd}
		\end{equation*}
		From \cite{BD} we know that $S$ admits a unique, up to isometries, primitive embedding in $L$ and this embedding has the lattice generated by an hyperplane section under the Plücker embedding $T\simeq \langle 6\rangle$ as orthogonal complement. We are interested in giving a relation between cubic threefolds and IHS manifolds of $K3^{[2]}$-type. This is classically (e.g. \cite{ACT}) done by looking at the cubic fourfolds which cover $\mathbb{P}^4$ and branch over a cubic threefold. Indeed, cubic fourfolds with an ordinary double point are described by Hassett in \cite[Section 4.2]{hassett_2000} as cubic fourfolds with discriminant six and they admit an associated $K3$ surface. Then, in Section 6 of \textit{loc. cit.}, the author gives a nice description of the Plücker divisor on a fourfold of $K3^{[2]}$-type associated to a generic cubic fourfold of discriminant six: it is $\theta=2\theta_{K3}-3\epsilon$ where $\theta_{K3}$ is a square six class on the associated $K3$ surface and $\epsilon$ is half of the exceptional class coming from the Hilbert--Chow morphism. In our case even though we do not ask for genericity we make the same choice; moreover, as they are all the same up to isometry, we choose $\theta_{K3}=3u_1+u_2$. So we have assigned an embedding $j: \langle 6\rangle\to L$ and $j(\langle 6\rangle)^\perp\simeq S$ as expected. Now using \cite[Cor~1.5.2]{Nik1980} (see also \Cref{teonik} below) we can extend the isometry $H^2_{\circ}(F(Y), \mathbb{Z})\oplus \langle 6\rangle\xrightarrow{\eta\ \oplus\ j} S(-1)\oplus j(\langle 6\rangle)$ to a marking $\bar{\eta}\colon H^2(F(Y), \mathbb{Z})\to L$. Finally, $\sigma\in \Aut(Y)$ induces on $F(Y)$ an automorphism that will be also denoted with $\sigma\in \Aut(F(Y))$ in order to simplify the notation. Therefore there exists a natural isometry on $L$ that is $\rho:=\bar{\eta}\circ\sigma^*\circ\bar{\eta}^{-1}$. So, $F(Y)$ is an IHS manifold of $K3^{[2]}$-type and admits a $(\rho, j)$-polarization with the lattice $\langle 6\rangle$ playing the role of $T$ in \Cref{def: rho_j polarizzazione} and a primitive third root of unity as $\zeta$. The period map $\mathcal{P}_{\langle 6\rangle}^{\rho, \zeta}$ relative to this space has the period domain isomorphic to a 10-dimensional complex ball \begin{equation*}
			\Omega^{\rho,\zeta}_{T}:=\left\{x\in\mathbb{P}(S_\zeta) | h_{S}(x,x)>0\right\}\simeq\mathbb{C}B^{10}.
		\end{equation*}
 
 \section{Degeneracy lattices} \label{DegSec}
	In this section we begin to study the degenerations of the automorphism $\rho$ over the nodal hyperplane. In order to do so we provide a general strategy which then will be used in each considered case. \newline

	We want to focus on the case described in \Cref{ex:fondante}. Take $\omega$ a period in $\mathcal{H}_\Delta$ and $(X, \eta)\in \mathcal{P}^{-1}_{\langle 6\rangle}(\omega)$ a point in the fiber of the period map of $\langle 6\rangle$-polarized IHS manifolds of $K3^{[2]}$-type. Remember that this is by \Cref{EstensioneACT} is the period of a nodal cubic threefold.
    \begin{rmk}\label{rmk:ample}
		As explained in \cite{BCS} the isometry $\rho\in O(L)$ is not represented by any automorphism of $X$. In fact if $\rho$ was represented by an automorphism of $X$, this one would be automatically non-symplectic. Consider now $l\in \NS(X)$ an ample class (it always exists as $X$ is projective) then consider $l+\rho^*l+(\rho^*)^2l$ which is still ample and invariant. Therefore it is  a multiple of the generator of the rank one invariant lattice, we call $\theta$ the primitive ample invariant class. Since $\delta_i\in S$, the divisor $\eta^{-1}(\delta_i)$ is orthogonal to $\theta$, yielding a contradiction as $(\eta^{-1}(\delta_i), \theta)>0$ by the ampleness of $\theta$. 
	\end{rmk}
    So, the main issue in finding an automorphism representing $\rho$ is the fact that there exist some MBM classes in the Nèron--Severi lattice which are orthogonal to the invariant lattice of the automorphism $\rho$. In order to deal with it we define the following.
	\begin{dfn}\label{dfn:degenlattice}
		The \emph{degeneracy lattice} of $(X, \eta)$ is the sublattice of $S$ generated by those MBM classes $\delta_i\in S$ which are orthogonal to $\omega$. 
	\end{dfn}
	This lattice is $\rho$-invariant and orthogonal to $j(\langle 6\rangle)$. So, in general, the degeneracy lattice will be $R_0:=\Span(\delta_1, \dots, \delta_n, \rho(\delta_1), \dots, \rho(\delta_n))$.
	\begin{rmk}\label{rmk: roots}
		For each $i\in\left\{1, \dots, n\right\}$ the sublattice $\Span(\delta_i, \rho(\delta_i))=:R_{\delta_i}\subset R_0$ is just the degeneracy lattice of a polarized IHS manifold $(X, \eta)$ generic in $\mathcal{H}_\Delta$ and, with a simple computation, $R_{\delta_i}\simeq A_2(-1)$. For the sake of completeness let us show this computation. By definition $R_{\delta_i}\simeq\langle \delta_i, \rho(\delta_i) \rangle$ with $\delta_i^2=(\rho(\delta_i))^2=-2$. Moreover, $(\delta_i, \rho(\delta_i))=(\rho(\delta_i), \rho^2(\delta_i))=(\rho(\delta_i), -\delta_i-\rho(\delta_i))=-(\rho(\delta_i), \delta_i)+2$ which implies $(\delta_i, \rho(\delta_i))=1$.
	\end{rmk} 
	
	Recall that in \Cref{ex:fondante} we provided an embedding $j: \langle 6\rangle\to L$ such that $j(\langle 6\rangle)^\perp\simeq S$. With the help of $j$ we can induce an embedding of $\langle 6\rangle \oplus R_0$ which in general will not be primitive. So, we define $T_0:=\overline{\langle 6\rangle \oplus R_0}$ as the saturation of $\langle 6\rangle \oplus R_0$ in $L$. Note that $T_0\hookrightarrow \Pic(X)$ by definition of degeneracy lattice and the equality will hold for a generic element, i.e. a $\langle 6\rangle$-polarized IHS fourfold of $K3^{[2]}$-type which has a generic period orthogonal to a fixed number of MBM classes $\delta_1, \dots, \delta_n$. \newline
	
	The strategy we will use in the following sections is similar to the one outlined in \cite{BCS} and in \cite[\S 11]{DolgachevKondo}. We outline it here. First, we will prove the following claim for each family having generically the nodal period considered in each case.
	\begin{claim} \label{claim}
		The Picard group of the Hilbert square $\hat{\Sigma}^{[2]}$ of the minimal resolution of the surface defined in \Cref{sec:cubiche} is generated by the square six polarization and $2i$ classes of square $(-2)$ orthogonal to it, with $i=1, \dots, 4$. Therefore it is generic in the above sense.
	\end{claim}	  
	Then we want to look for an isometry on $L$ related to $\rho$ which has a bigger invariant lattice. In particular, denoting with $S_0:= T_0^{\perp}$ in $L$, we look at $\id_{T_0}\oplus\rho_{|S_0} \in O(T_0)\oplus O(S_0)$. If it can be lifted to an isometry $\rho_0\in O(L)$ then as $T_0\simeq \Pic(\hat{\Sigma}^{[2]})$ we can find an IHS manifold of $K3^{[2]}$-type with the same period (by definition of period of a nodal cubic given in \cite{ACT}) generic in the space of the IHS manifolds of $K3^{[2]}$-type which are $(\rho_0, j)$-polarized. It is important to remark that by definition of the $(\rho, j)$-polarization, as the isometry $\rho \in O(L)$ comes from a  non-symplectic automorphism $\sigma$ on $X$, it can be restricted to an isometry $\rho_{S_0}\in O(S_0)$. Indeed, as $\omega\in \delta^{\perp}$ we deduce that $\zeta_3\cdot\left(\rho\left(\delta\right),\omega\right)=\left(\rho\left(\delta\right),\rho\left(\omega\right)\right)=\left(\delta,\omega\right)=0$, thus the isometry $\rho$ can be restricted to an isometry of both the Picard lattice and its orthogonal complement. In order to lift the isometry we will apply the following result
	\begin{prop}[\protect{\cite[Cor~1.5.2]{Nik1980}}]\label{teonik}
		Let $L$ be a finite index overlattice of $S\oplus T\subset L$ determined by the pair $(H, \gamma)$, where $H<D_S$ is a subgroup and $\gamma: H\to D_T$ is a group monomorphism. Moreover, let $\rho_S\in O(S)$ and $\rho_T\in O(T)$ be two isometries such that the induced isometry on the discriminant $\rho^*_S\in O(D_S)$ restricts to an isometry of $H$.\newline Then the isometry $\rho_S\oplus\rho_T\in O(S)\oplus O(T)$ lifts to an isometry $\rho\in O(L)$ of $L$ if and only if $\rho_S|_H$ is conjugate to the induced isometry $\rho_T^*$ via $\gamma$, or equivalently $\gamma\circ\rho_S^*|_H=\rho^*_T\circ\gamma$.
	\end{prop} 
	\begin{proof}
		The first implication is obvious as $\rho_S=\rho|_S$ and $\rho_T=\rho|_T$. The second implication is done just by noting that in the diagram 
		\begin{equation*}
			\begin{tikzcd}
				H \ar[r, "\rho_S^*"] \ar[d, "\gamma"] & H \ar[d, "\gamma"] \\
				\gamma(H) \ar[r, "\rho_T^*"] & \gamma(H)	
			\end{tikzcd}
		\end{equation*}
		the pair $\left(\rho_S^*(H)=H, \rho_T^*\circ\gamma\circ(\rho_S^*)^{-1}=\gamma\right)$ determines an overlattice isometric to $L$ and we call this isometry $\rho$.
	\end{proof}
	Recall now that the Picard group of an IHS fourfold of $K3^{[2]}$-type can be written as $$\Pic(\hat{\Sigma}^{[2]})\simeq\Pic(\hat{\Sigma})\oplus\langle\epsilon\rangle\simeq T_{K3}\oplus\langle -2 \rangle$$ where $\epsilon$ is the $(-2)$-class given by half of the exceptional divisor introduced by the Hilbert--Chow morphism and $T_{K3}\subset L_{K3}$ is a sublattice of the $K3$ lattice $L_{K3}\simeq U^{\oplus 3}\oplus E_8(-1)^{\oplus 2}$. Thus we consider the isometry $\tilde{\rho_0}:=\rho_0|_{T_{K3}}$. Then the following proposition proves that there exists an automorphism $\sigma$ of the $K3$ surface $\hat{\Sigma}$ whose action on cohomology is conjugate to $\tilde{\rho_0}$.
	\begin{prop}\label{Prop:Namikawa}
		Let $\Sigma$ be a $K3$ surface and $\rho\in O(L_{K3})$. Suppose that $\rho(\omega)=\lambda\omega$ where $\omega\in H^2(\Sigma, \mathbb{C})$ is the period of $\Sigma$ and $1\neq\lambda\in\mathbb{C}^*$. Then if $\Pic(\Sigma)$ is fixed by $\rho$ there exists $\sigma\in\Aut(\Sigma)$ and an element $w$ in the Weyl subgroup of $\Sigma$ such that $w\rho w^{-1}=\sigma$.
	\end{prop}
	This proposition is an immediate corollary of the following theorem by Namikawa.
	\begin{thm}[\cite{Namikawa} Theorem 3.10]\label{thm:namikawa}
		Let $\Sigma$ be a $K3$ surface and $G$ a finite subgroup of the group of isometries in $\Lambda=H^2(\Sigma,\mathbb{Z})$. Denote by $\omega$ the period of $\Sigma$, by $\Lambda^G$ the sublattice of elements in $\Lambda$ fixed by $G$ and set $S_{G,\Sigma}=(\Lambda^G)^{\perp}\cap \left\{\mathbb{C}\omega\right\}^\perp=(\Lambda^G)^{\perp}$ in $H^{1,1}_\mathbb{Z}(\Sigma)$. Then there exists an element $t$ in the Weyl subgroup of $\Sigma$, $W(\Sigma)$, such that $tGt^{-1}\subset \Aut(\Sigma)$ if and only if
		\begin{enumerate}[i)]
			\item $\mathbb{C}\omega$ is $G$-invariant;
			\item $S_{G,\Sigma}$ contains no element of length $-2$;
			\item if $\omega\in \Lambda^G$ then $S_{G,\Sigma}$ is either $0$ or nondegenerate
			and negative definite;
			\item[iii')] if $\omega\notin \Lambda^G$ then $\Lambda^G$ contains an
			element $a$ with $(a,a)>0.$
		\end{enumerate}
	\end{thm}
	Therefore if we consider the natural automorphism on $\Sigma^{[2]}$ induced by $\sigma$ then we see that the former variety is equipped with an automorphism whose action in cohomology is conjugate to $\rho_0$.
	\begin{thm}\label{thm:teoremino}
		Let $C$ be a nodal cubic and $\hat{\Sigma}, \, \theta$ as in \Cref{sec:cubiche}. Let $(\hat{\Sigma}^{[2]}, \eta)$ the Hilbert square of $(\hat{\Sigma}, \Tilde{\eta})$ and suppose that $\Pic(\hat{\Sigma}^{[2]})\simeq j(\theta)\oplus W\simeq\overline{\langle 6\rangle \oplus R_0}=T_0$. If the action induced by $\rho$ on the discriminant group $D_W$ is trivial then the isometry $\id_{T_0}\oplus\rho_{|S_0} \in O(T_0)\oplus O(S_0)$ lifts to an isometry $\rho_0\in O(L)$. Finally, if we define $K(T_0)$ as the chamber containing a Kähler class of $\Sigma^{[2]}$, then the latter is $K(T_0)$-general and defines a point in $\mathcal{M}^{\rho_0, \zeta}_{K(T_0)}$.
	\end{thm}
	\begin{proof}
		In order to prove the statement it is sufficient to prove that the isometry $\text{id}_{T_0}\oplus\rho_{|S_0} \in O(T_0)\oplus O(S_0)$ lifts to an isometry $\rho_0\in O(L)$. In order to lift it the condition stated in \Cref{teonik} is that the action induced by $\rho$ on the discriminant group $D_{S_0}$ is trivial. This discriminant group is isomorphic to a quotient of $D_{W}\oplus D_{S(-1)}$, but the action of $\rho$ is trivial on $D_W$ by hypothesis, therefore it is enough to check the triviality on $D_{S(-1)}$. Again $D_{S(-1)}$ is isomorphic to a quotient of $D_L\oplus D_{\langle 6\rangle}$, the action induced by $\rho$ is trivial on the first factor because it is an order three isometry on $D_L\simeq \mathbb{Z}/2\mathbb{Z}$ and on the second by construction. The $K(T_0)$-generality of $\hat{\Sigma}^{[2]}$ follows from \cite[Lemma 5.2]{BCS_complex}.
	\end{proof}
	\begin{rmk}
		In this section we provided also an automorphism on $\hat{\Sigma}$ whose action on cohomology is conjugate to $\tilde{\rho_0}$. Therefore, we proved that, under the same hypotheses of \Cref{thm:teoremino}, $(\hat{\Sigma}, \Tilde{\eta})\in\mathcal{K}^{\tilde{\rho_0}, \zeta}_{T_{K3}}$, the moduli space of $(\tilde{\rho_0}, \tilde{j})$-polarized $K3$ surfaces.
	\end{rmk}
	In the next sections we analyze case by case what happens. In order to perform this analysis we define $\Delta_3^{A_i}$ for $i=1, \dots, 4$ as the biggest sub-locus of the space of stable cubic threefolds where the generic element is a cubic with an isolated singularity of type $A_i$.
	\begin{rmk}
		One can check (an explicit computation is done in \Cref{appendix}) that the dimension of the $\Delta_3^{A_i}$ locus is $10-i$.
	\end{rmk}
	\begin{notation}
		In the following sections every time an IHS manifold endowed with a $(\rho, j)$-polarization will appear we will use the following notations. The definition of the isometry playing the role of $\rho$ will be clear in all cases. The embedding $j$ will be constructed in the same way each time as follows. We will exhibit a square six class $\theta$ which has an embedding in $L$ as in \Cref{sec:moduli}. This embedding will induce an embedding of the whole fixed lattice in $L$ playing the role of $j$. Moreover, we will choose for any Hilbert square over a $K3$ surface the natural marking induced by a marking on a $K3$ surface, i.e. fixing a marking $\tilde{\eta}$ on a $K3$ surface $\hat{\Sigma}$ the natural marking $\eta$ on its Hilbert square $\hat{\Sigma}^{[2]}$ induced by $\tilde{\eta}$ is the lifting of $\tilde{\eta}\oplus\text{id}_{\langle -2\rangle}$ to the whole second cohomology group with integer coefficients. Then we will take the embedding $\iota$ in a way compatible with $j$, i.e. $\iota(\langle 6\rangle)= 2\theta_{K3}-3\epsilon$ with $\theta_{K3}$ fixed and this embedding will induce also an embedding on its orthogonal. In order to lighten the notation and the exposition we will often omit the markings where their presence is obvious.
	\end{notation}
 \section{Singularity of type $A_1$}\label{sec:A1}
	In this section we recall the results discussed in \cite[Section 4.3]{BCS} where the authors provide a birationality result for the nodal hyperplane. In order to see the analogy with the loci we are interested in we briefly review it giving a proof that fits our framework.\newline

	With the notation of \Cref{sec:cubiche} let $p_0$ be an isolated singularity of type $A_1$ for $C\subset\mathbb{P}^4$. This implies, for reasons of corank of the singularity (see \cite[Section 2]{AllcockSing} or directly do the computation), that $f_2$ is a rank 4 quadratic form and $Y\subset\mathbb{P}^5$ has an isolated singularity of type $A_2$. Moreover, by genericity, we can assume that the surface $\Sigma$ given by the following equations:
	\begin{equation}
		\begin{cases}
			f_2(x_1,x_2,x_3,x_4)=0 \\
			f_3(x_1,x_2,x_3,x_4)+x_5^3=0.
		\end{cases}
	\end{equation}
	is a smooth $K3$ surface. Here $x_1, \dots, x_5$ are the homogeneous coordinates on the hyperplane $H\subset \mathbb{P}^5$ given by $\{x_0=0\}$. The covering automorphism $\sigma$ on $Y$ induces an automorphism $\tau$ of $\Sigma$. Explicitly, $\tau$ is given by $x_5\mapsto\zeta_3\cdot x_5$ and the identity on the other coordinates. The fixed locus of $\tau$ is a curve of genus 4. In \cite{artebanisarti} the authors show that the generic case (i.e. the one which we are considering) has $\Pic(\Sigma)\simeq U(3)$. Therefore for the Hilbert square of $\Sigma$ it holds:
	\begin{equation*}
		\Pic(\Sigma^{[2]})\simeq \Pic(\Sigma)\oplus\langle -2\rangle\simeq U(3)\oplus\langle -2\rangle\simeq \langle 6\rangle\oplus A_2(-1):= T_0.
	\end{equation*}
	Note that the last isometry is given by:
	\begin{equation*}
		U(3)\oplus\langle -2\rangle:=\langle u_1,u_2\rangle\oplus\langle \epsilon\rangle\simeq\langle 2(u_1+u_2)-3\epsilon\rangle\oplus\langle u_1-\epsilon, \epsilon-u_2\rangle\simeq j(\theta)\oplus W.
	\end{equation*}
	This proves \Cref{claim}. Moreover, $\tau$ induces an order three isometry $\rho\in O(L)$ with $j(\theta)$ as fixed sublattice. So, $\rho$ restricts to an order three isometry without fixed points of $W\simeq A_2(-1)$. There exists only one isometry of $A_2(-1)$ of order three without fixed points modulo conjugation and its action on the discriminant $D_{A_2(-1)}$ is trivial. Therefore, applying \Cref{thm:teoremino} there exists $\rho_0\in O(L)$ lifting $\id_{|T_0}\oplus\rho_{|S_0}$ and $(\Sigma^{[2]}, \eta, \tau)\in \mathcal{M}^{\rho_0, \zeta}_{K(T_0)}$. \begin{prop}[ \protect{\cite[Proposition~4.6]{BCS}}]\label{prop:A1}
		Let $R_0=\Span(\delta_1, \rho(\delta_1))\simeq A_2(-1)$ be the degeneracy lattice relative to a period $\omega$ of a generic cubic threefold having a single singularity of type $A_1$. Then the $A_1$ locus $\Delta^{A_1}_3$ is birational to the moduli space $\mathcal{N}^{\rho_0, \zeta}_{K(T_0)}$.
	\end{prop}
	\begin{proof}
		The idea is to use the same structure of the proof of \cite[Proposition 4.6]{BCS} within our theoretical approach. First we note that the extension of the period map $\mathcal{P}^3:\mathcal{C}^{\text{sm}}_3\to\frac{\mathbb{B}^{10}\setminus\left(\mathcal{H}_n\cup \mathcal{H}_c\right)}{\mathbb{P}\Gamma}$ to the nodal locus is done by \cite[Section 6]{ACT} defining its period as the period of its associated $K3$ surface. In our case the generic $A_1$ nodal cubic has the period 
		\begin{equation*}
			\mathcal{P}^3(C):= \mathcal{P}^{\tilde{\rho_0}, \zeta}_{U(3)}((\Sigma, \eta)).
		\end{equation*}
		The latter period map is defined by taking $\mathcal{K}^{\tilde{\rho_0}, \zeta}_{U(3)}$ as the moduli space of lattice polarized $K3$ surfaces with a non-symplectic automorphism of order three whose action on cohomology is conjugate to $\tilde{\rho_0}$. Following \cite{DolgachevKondo} this space comes with a period map 
		\begin{equation*}
			\mathcal{P}^{\tilde{\rho_0}, \zeta}_{U(3)}:\mathcal{K}^{\tilde{\rho_0}, \zeta}_{U(3)}\to\Omega^{\tilde{\rho_0}, \zeta}_{U(3)}:=\left\{x\in \mathbb{P}((S_0)_\zeta)\mid h_{S_0}(x,x)>0\right\}
		\end{equation*}
		which induces a bijection 
		\begin{equation*}
			\mathcal{P}^{\tilde{\rho_0}, \zeta}_{U(3)}:\mathcal{K}^{\tilde{\rho_0}, \zeta}_{U(3)}\to\frac{\Omega^{\tilde{\rho_0}, \zeta}_{U(3)}\setminus\mathcal{H}_{U(3)}}{\Gamma^{\tilde{\rho_0}, \zeta}_{U(3)}}
		\end{equation*}
		where we denote with $$\mathcal{H}_{U(3)}:=\bigcup\limits_{\mu\in S_0, \ \mu^2=-2} \mu^{\perp}\cap\Omega^{\tilde{\rho_0}, \zeta}_{U(3)}$$ and with $$\Gamma^{\tilde{\rho_0}, \zeta}_{U(3)}:=\left\{\gamma\in O(L_{K3})\mid \gamma\circ\tilde{\rho_0}=\tilde{\rho_0}\circ\gamma\right\}.$$ By definition we find the following equality:
		\begin{equation*}
			\frac{\Omega^{\tilde{\rho_0}, \zeta}_{U(3)}}{\Gamma^{\tilde{\rho_0}, \zeta}_{U(3)}}=\frac{\Omega^{\rho_0, \zeta}_S\cap \delta_1^{\perp}}{\Gamma^{\rho_0, \zeta}_S}.
		\end{equation*}
		As proven above $(\Sigma^{[2]}, \eta, \tau)$ defines a point in $\mathcal{M}^{\rho_0, \zeta}_{K(T_0)}$. The period map in this space, following equation (\Cref{eq: Moduli}), descends to a bijection
		\begin{equation*}
			\mathcal{P}^{\rho_0, \zeta}_{T_0}:\mathcal{N}^{\rho_0, \zeta}_{K(T_0)}=\frac{\mathcal{M}^{\rho_0, \zeta}_{K(T_0)}}{\text{Mon}^2(T_0, \rho_0)}\to\frac{\Omega^{\rho_0, \zeta}_{T_0}\setminus\left(\mathcal{H}_{T_0} \cup \mathcal{H}^{'}_{T_0} \right)}{\Gamma^{\rho_0, \zeta}_{T_0}}.
		\end{equation*}
		By their definitions we see that $\Omega^{\rho_0, \zeta}_{T_0}=\Omega^{\tilde{\rho_0}, \zeta}_{U(3)}$ and $\Gamma^{\rho_0, \zeta}_{T_0}=\Gamma^{\tilde{\rho_0}, \zeta}_{U(3)}$. Moreover, as $S_0\subset L_{K3}$ is a sublattice of the unimodular $K3$ lattice the following holds: $$\mathcal{H}_{T_0}:=\bigcup\limits_{\mu\in S_0, \ \mu^2=-2} \mu^{\perp}\cap\Omega^{\rho_0, \zeta}_{T_0}=\mathcal{H}_{U(3)}.$$  Now it is easy to see that $\frac{\Omega^{\rho_0, \zeta}_{T_0}\setminus\left(\mathcal{H}_{T_0} \cup \mathcal{H}^{'}_{T_0} \right)}{\Gamma^{\rho_0, \zeta}_{T_0}}$ is birational to $\frac{\Omega^{\tilde{\rho_0}, \zeta}_{U(3)}\setminus\mathcal{H}_{U(3)}}{\Gamma^{\tilde{\rho_0}, \zeta}_{U(3)}}$. So in order to conclude the proof it is enough to show that $\Delta^{A_1}_3$ is birational to $\mathcal{K}^{\tilde{\rho_0}, \zeta}_{U(3)}$ as the claim of the proposition would follow through a composition of birational morphisms. But this is true as in \cite{ACT} the authors show that the discriminant locus maps isomorphically to its image, the nodal hyperplane arrangement, through the period map, therefore a generic point in $\Delta^{A_1}_3$ is mapped isomorphically to the period of a generic $K3$ surface in $\mathcal{K}^{\tilde{\rho_0}, \zeta}_{U(3)}$ and by \cite{DolgachevKondo} this is an isomorphism. 
	\end{proof}
	\section{Singularity of type $A_3$ and $A_4$}\label{sec:A3A4}
	In this section we prove \Cref{thm:main} for $\Delta_3^{A_3}$ and $\Delta_3^{A_4}$. The cases treated in this section are very similar to the $A_1$ case; we will keep the same notation.\newline
    
	Let $p_0$ be a singularity of type $A_k$ with $k=3,4$ for $C\subset\mathbb{P}^4$. Remember that if $k\geq 2$ the singularities have corank 1, therefore this translates in $f_2$ being a rank 3 quadratic form.
	
	\subsection{Singularity of type $A_3$}
	If $C$ has an singularity of type $A_3$ then with a standard computation (just add a cube in the new variable) one shows that $Y$ has an singularity of type $E_6$, thus, this time $\Sigma$ has one isolated singularity of type $A_5$, call it $p\in\Sigma$. Let us consider $\hat{\Sigma}$, the minimal resolution of $\Sigma$ which is a sequence of blow ups of $\Sigma$ at $p$. This is a $K3$ surface. The covering automorphism $\sigma$ on $Y$ descends to an automorphism $\tau$ on $\Sigma$. So, the singular point $p$ is a fixed point for the automorphism $\tau$. As the locus we are blowing up each time is fixed by the automorphism there exists a unique lift $\hat{\tau}$ such that $\hat{\tau}$ is an automorphism of $\hat{\Sigma}$ and commutes with the resolution map $\beta$. The fixed locus for $\hat{\tau}$ consists of two curves and two points. Stegmann in her PhD thesis gives a detailed description of surfaces which are complete (2,3)-intersection in $\mathbb{P}^4$ and from \cite[Proposition 6.3.10]{Stegmann} we find that the Picard lattice of $\hat{\Sigma}$ is just the span of $\langle C_1, E_1, E_2, E_3, E_4, E_5\rangle\subset \Pic(\hat{\Sigma})$ with the following Gram matrix: \begin{equation*}
		\begin{pmatrix*}[r]
			0 &  0 &  0 &  1 & 0 & 0\\
			0 & -2 &  1 &  0 & 0 & 0\\
			0 &  1 & -2 &  1 & 0 & 0\\
			1 &  0 &  1 & -2 & 1 & 0\\
			0 &  0 &  0 &  1 &-2 & 1\\
			0 &  0 &  0 &  0 & 1 & -2
		\end{pmatrix*}
	\end{equation*}
	Geometrically, the divisors $E_1, \dots, E_5$ are the $(-2)$-curves coming from the successive blow ups of $\Sigma$. By genericity the Picard lattice has rank 6 (see \Cref{appendix} for calculations), therefore $$\langle C_1, E_1, E_2, E_3, E_4, E_5\rangle= \Pic(\hat{\Sigma}).$$
	As the fixed locus of $\hat{\tau}$ consists of two curves and two isolated points, we deduce, using \cite[Table 2]{artebanisarti}, that the Picard lattice admits an embedding of the invariant sublattice $U\oplus A_2(-1)^{\oplus 2}$. Moreover, this is an isomorphism of lattices. In fact, $\langle C_1+E_3, C_1\rangle \oplus \langle E_1, E_2-C_1\rangle \oplus \langle E_4-C_1, E_5\rangle\simeq U\oplus A_2(-1)^{\oplus 2}$ exhibits $U\oplus A_2(-1)^{\oplus 2}$ as a primitive sublattice of the Picard lattice. Now we can compute the determinants and notice that they are both 9, so the lattices are isomorphic. Moreover the automorphism $\hat{\tau}$ is clearly non-symplectic (as it is the multiplication by a third root of unity $\zeta_3$ of the last coordinate).
	Now we investigate some properties of $\hat{\Sigma}^{[2]}$. Its Picard lattice is $$\Pic(\hat{\Sigma}^{[2]})\simeq\Pic(\hat{\Sigma})\oplus \langle -2\rangle\simeq  U\oplus A_2(-1)^{\oplus 2}\oplus\langle -2\rangle\simeq E_6(-1)\oplus \langle 6 \rangle:=T_0$$ where the last isomorphism is given by \begin{equation*}
		\begin{split}
			&\Pic(\hat{\Sigma})\oplus \langle -2\rangle \simeq \\ &\simeq \langle E_1, E_2, E_3, E_4, E_5, C_1- \epsilon\rangle \oplus \langle 2(2C_1+E_1+2E_2+3E_3+2E_4+E_5)- 3\epsilon\rangle
		\end{split}
	\end{equation*} where $\epsilon$ is the $(-2)$-class which is half of the divisor introduced by the Hilbert--Chow morphism. It is useful to describe also the trascendental lattice $\TR(\hat{\Sigma}^{[2]})$ of $\hat{\Sigma}^{[2]}$.
	\begin{equation*}
		\TR(\hat{\Sigma}^{[2]})\simeq \TR(\hat{\Sigma})\simeq U^{\oplus 2}\oplus E_8(-1)\oplus A_2(-1)^{\oplus 2}
	\end{equation*}
	First we note that $\Pic(\hat{\Sigma}^{[2]})\simeq E_6(-1)\oplus \langle 6 \rangle\simeq W\oplus j(\theta)$ with $\theta=2\theta_{K3}-3\epsilon$ and again we assume  $\theta_{K3}=3u_1+u_2$ with an isometry so with another choice of the marking. This proves \Cref{claim}. Again we can associate to $\hat{\tau}$ an order three isometry $\rho\in O(L)$ with $j(\theta)$ as fixed sublattice with \Cref{Prop:Namikawa}. Therefore it restricts to an order three isometry without fixed points on $W\simeq E_6(-1)$. As said in \Cref{isomE_6} there exists only one, up to conjugation and sign, order three isometry without fixed points on $E_6$. We can write it explicitly. Let $E_6=\langle e_1, \dots, e_6\rangle$ with the standard Bourbaki numeration
	\begin{equation*}
		\begin{tikzcd}
			e_1 \ar[r, dash] & e_3 \ar[r, dash]& e_4 \ar[r, dash] & e_5 \ar[r, dash] & e_6 \\
			&& e_2 \ar[u, dash] &&
		\end{tikzcd}
	\end{equation*}
	then $\rho$ is given by
	\begin{align*}
		e_1 &\mapsto -e_1-e_3-e_4-e_5-e_6 \\
		e_2 &\mapsto e_3+e_4+e_5 \\
		e_3 &\mapsto-e_2-e_3-e_4 \\
		e_4 &\mapsto -e_4-e_5 \\
		e_5 &\mapsto e_4 \\
		e_6 &\mapsto e_1+e_2+ e_3+e_4+e_5.
	\end{align*}
	Now, $\mathbb{Z}/3\mathbb{Z}\simeq D_W=\langle-\frac{4}{3}e_1-e_2-\frac{5}{3}e_3-2e_4-\frac{4}{3}e_5-\frac{2}{3}e_6\rangle$ and therefore after substituting the expression of $\rho$ on the generator of $D_W$ we note that the action induced by $\rho$ is trivial on $D_W$. Finally we can apply \Cref{thm:teoremino} to deduce also in this case that there exists $\rho_0\in O(L)$ lifting $\id_{|T_0}\oplus\rho_{|S_0}$ and  $(\hat{\Sigma}^{[2]}, \eta, \hat{\tau})\in\mathcal{M}^{\rho_0, \zeta}_{K(T_0)}$.
	\begin{prop}\label{prop:A3}
		Let $R_0=\Span(\delta_1, \delta_2, \delta_3, \rho(\delta_1), \rho(\delta_2), \rho(\delta_3))\simeq E_6(-1)$ be the degeneration lattice relative to a period $\omega$ of a generic cubic threefold having one singularity of type $A_3$. Then the $A_3$ locus $\Delta^{A_3}_3$ is birational to the moduli space $\mathcal{N}^{\rho_0, \zeta}_{K(T_0)}$.
	\end{prop}
	\begin{proof}
		It is easy to see that one can adjust the same proof of \Cref{prop:A1} to this case and everything works.
	\end{proof}
	\subsection{Singularity of type $A_4$}
	This section is analogous to the previous section so we omit the details. \newline
	If $C$ has now a singularity of type $A_4$ then $Y$ has a singularity of type $E_8$ and $\Sigma$ an isolated singularity of type $E_7$ which we will call $p\in\Sigma$. Stegmann's work \cite[Proposition 6.3.12]{Stegmann} and the computations in \Cref{appendix} provide us once again the description of its Picard lattice $\Pic(\hat{\Sigma})=\langle C_1, E_1, \dots E_7\rangle$ with the following Gram matrix: 
	\begin{equation*}
		\begin{pmatrix*}[r]
			0 &  1 &  0 &  0 & 0 & 0 & 0 & 0\\
			1 & -2 &  1 &  0 & 0 & 0 & 0 & 0\\
			0 &  1 & -2 &  1 & 0 & 0 & 0 & 0\\
			0 &  0 &  1 & -2 & 1 & 0 & 0 & 0\\
			0 &  0 &  0 &  1 &-2 & 1 & 0 & 1\\
			0 &  0 &  0 &  0 & 1 & -2& 1 & 0\\
			0 &  0 &  0 &  0 & 0 & 1 & -2& 0\\
			0 &  0 &  0 &  0 & 1 & 0 & 0 & -2
		\end{pmatrix*}
	\end{equation*}
	The fixed locus of the automorphism $\hat{\tau}$, induced by the cover automorphism, is just the union of three isolated points and three curves. Therefore we deduce from \cite{artebanisarti} that there exists an embedding of $U\oplus E_6(-1)$ in the Picard lattice which turns out to be an isometry. Indeed, it is immediate to see that $\Pic(\hat{\Sigma})\simeq\langle C_1, C_1+E_1\rangle\oplus\langle E_2-C, E_3, \dots E_7\rangle\simeq U\oplus E_6(-1)$ and they both have determinant 3.

	Now $\hat{\Sigma}^{[2]}$ has Picard lattice
	\begin{equation*}
		\Pic(\hat{\Sigma}^{[2]})\simeq\Pic(\hat{\Sigma})\oplus\langle -2\rangle\simeq U\oplus E_6(-1)\oplus\langle -2\rangle\simeq E_8(-1)\oplus\langle 6\rangle
	\end{equation*}
	where the last isometry is given by 
	\begin{equation*}
		\begin{split}
			\Pic(\hat{\Sigma})&\oplus \langle -2\rangle \simeq \langle E_1, \dots, E_7, C_1- \epsilon\rangle \oplus \\ &\oplus\langle 2(-2C_1+3E_1+4E_2+5E_3+6E_4+4E_5+2E_6+3E_7)- 3\epsilon\rangle
		\end{split}
	\end{equation*}
	and $\epsilon$ is, as usual, the $(-2)$-class given by half of the divisor introduced by the Hilbert--Chow morphism. Once again we note that with \Cref{Prop:Namikawa} we have the existence of an order 3 automorphism on $\hat{\Sigma}$ conjugate to $\tilde{\rho}=\rho_{|T_{K3}}$. This automorphism is, modulo conjugation, uniquely determined by its invariant lattice by \cite{artebanisarti}, therefore it is $\hat{\tau}$. As in the previous section we note that $\Pic(\hat{\Sigma}^{[2]})\simeq j(\theta)\oplus W$, proving \Cref{claim}. Moreover $W\simeq E_8(-1)$ which has trivial discriminant, therefore we can apply \Cref{thm:teoremino} and check that $(\hat{\Sigma}^{[2]}, \eta, \hat{\tau})\in \mathcal{M}^{\rho_0, \zeta}_{K(T_0)}$.
	Then we can deduce the following
	\begin{prop}\label{prop:A4}
		Let $R_0=\Span(\delta_1, \dots, \delta_4, \rho(\delta_1), \dots, \rho(\delta_4))\simeq E_8(-1)$ be the degeneracy lattice relative to a period $\omega$ of a generic cubic threefold having one singularity of type $A_4$. Then the $A_4$ locus $\Delta^{A_4}_3$ is birational to the moduli space $\mathcal{N}^{\rho_0, \zeta}_{K(T_0)}$.
	\end{prop}
	\section{Singularity of type $A_2$}\label{sec: A2}
	In this Section we begin the study of $\Delta_3^{A_2}$. As we will see in this section this case is quite different from the other cases already treated, nevertheless we will use the same notation in order to stress the similarities. Indeed, in this case the associated $K3$ has three isolated singularities instead of one.\newline

	With a standard computation (using e.g the recognition principle stated in \cite{BruceWall}) we can see that when $C$ has an isolated singularity of type $A_2$, $Y$ has an isolated singularity of type $D_4$. The family of cubic threefolds having a singularity of type $A_2$ corresponds to the generic family of complete $(2,3)$-intersections in $\mathbb{P}^4$ where the quadratic part $f_2$ has rank 3 (a proof of this fact can be found in the \Cref{appendix}). Thus, the singular locus of the quadric hypersurface defined by $f_2=0$ is a line $(l_1(t):\dots:l_4(t):s)$; this line intersects the cubic hypersurface in 3 points $p_1, p_2, p_3$ for $f_3$ sufficiently generic.  Hence, $\Sigma$ has exactly three singular points. Let $\sigma$ be a covering automorphism on $Y$; it restricts to an order three automorphism $\tau$ on $\Sigma$, namely the one which maps $x_5\mapsto \zeta_3\cdot x_5$ where $\zeta_3$ is just a primitive third root of unity. As we are interested in the minimal resolution $\hat{\Sigma}$ of $\Sigma$ (which is just the blow up of $\Sigma$ at each $p_i$) we want to prove that the automorphism $\tau$ of $\Sigma$ lifts to an automorphism $\hat{\tau}$ of $\hat{\Sigma}$ and find its fixed locus.
	\begin{prop}\label{prop:automorfismo A2}
		With the above notation there exists a unique lift $\hat{\tau}$ such that $\hat{\tau}$ is an automorphism of $\hat{\Sigma}$ and commutes with the map giving the minimal resolution of singularities $\beta$. Moreover $\Fix(\hat{\tau})\simeq\Fix(\tau)$ is a curve of genus 4.
	\end{prop} 
	\begin{proof}
		The singular locus $\Sing(\Sigma)$ is a proper orbit for the automorphism $\tau$, where proper means that $\Fix(\tau)\cap\Sing(\Sigma)=\emptyset$. Indeed the points in the singular locus are permuted by $\sigma$. Take $p_i\in\Sing(\Sigma)$ of coordinates $(\bar{l_1}:\dots:\bar{l_4}:\bar{s})$ then $(\bar{l_1}:\dots:\bar{l_4}:\zeta_3\bar{s})$ is again on the same line defined by the singular locus of $f_2=0$ and a different point of $\Sing(\Sigma)$. So $\Sing(\Sigma)$ is mapped to itself and therefore there exists a unique lift $\hat{\tau}$ such that $\hat{\tau}$ is an automorphism of $\hat{\Sigma}$ and commutes with $\beta$. By construction of $\hat{\tau}$, $\Fix(\hat{\tau})\cap(\hat{\Sigma}\setminus\beta^{-1}(\Sing(\Sigma)))\simeq\Fix(\tau)$ as $\hat{\tau}$ acts in the same way of $\tau$ outside the exceptional divisors introduced by blowing up. Moreover, the diagram
		\begin{equation*}
			\begin{tikzcd}
				\hat{\Sigma} \ar[r, "\hat{\tau}"] \ar[d, "\beta", dashed] & \hat{\Sigma} \ar[d, "\beta", dashed] \\
				\Sigma \ar[r, "\tau"] & \Sigma
			\end{tikzcd}
		\end{equation*} commutes by the universal property of blow-ups, therefore $\hat{\tau}(\beta^{-1}(p_i))=\beta^{-1}(\tau (p_i))=\beta^{-1}(p_j)$ (where $j=i+1$ modulo 3).
		Therefore $\Fix(\hat{\tau})\simeq\Fix(\tau)$ which is a curve of genus 4. In order to see that it is indeed a genus 4 curve one can look at the explicit computation done in \cite[Proposition 4.7]{artebanisarti}.
	\end{proof} 
	It follows from the results in \textit{loc. cit.}, since the automorphism $\hat{\tau}$ on the $K3$ surface $\hat{\Sigma}$ fixes exactly one curve, that the invariant lattice $T(\hat{\tau})$ in $H^2(\hat{\Sigma}, \mathbb{Z})$ is isometric to $U(3)$ and its orthogonal complement in  $H^2(\hat{\Sigma}, \mathbb{Z})$ is isometric to $U \oplus U(3) \oplus E_8(-1)^{\oplus 2}$. As we are considering a generic $Y$ with only one singularity of type $D_4$ and thus a generic K3 surface $\Sigma$ with exactly three $A_1$ singularities, the Picard lattice $T$ of $\hat{\Sigma}$ is of rank four. We use again the description given in \cite[Proposition 6.3.8]{Stegmann} and we find a lattice $\langle C_1, E_1, E_2, E_3\rangle\subset \Pic(\hat{\Sigma})$ with the following Gram matrix: \begin{equation*}
		\begin{pmatrix*}[r]
			0 &  1 &  1 &  1\\
			1 & -2 &  0 &  0\\
			1 &  0 & -2 &  0\\
			1 &  0 &  0 & -2
		\end{pmatrix*}
	\end{equation*}
	We consider $C$ to be generic with a singularity of type $A_2$ therefore by genericity (using the computations of the rank of the Picard group in \Cref{appendix}) it holds $\langle C_1, E_1, E_2, E_3\rangle= \Pic(\hat{\Sigma})$. 
	\begin{rmk}
		Note that $$\langle C_1, E_1, E_2, E_3\rangle=\langle C_1, C_1+E_1, -2C_1-E_1+E_2, -E_2+E_3 \rangle \simeq U\oplus A_2(-2).$$ From the discussion above we know that $T(\hat{\tau})\simeq U(3)$ admits a primitive embedding $U(3)\hookrightarrow T$, so from this isomorphism we see that  $U(3)\oplus A_2(-2)\subset U\oplus A_2(-2)$ can be written as a sublattice of finite index.
	\end{rmk}
	As in previous sections we note the following isometry: 
	\begin{equation*}
		\Pic(\hat{\Sigma}^{[2]})\simeq \Pic(\hat{\Sigma})\oplus\langle -2\rangle\simeq U\oplus A_2(-2)\oplus \langle -2\rangle \simeq \langle 6\rangle \oplus D_4(-1).
	\end{equation*}
	This isometry can be described by \begin{equation*}
		\Pic(\hat{\Sigma}^{[2]})= \langle C_1, E_1, E_2, E_3, \epsilon\rangle\simeq \langle-E_1, E_2, C_1-\epsilon, E_3, 2(2C_1+E_1+E_2+E_3)-3\epsilon\rangle.
	\end{equation*}
	Here we are implicitly giving an embedding of the square six polarization class $\theta$ into the Picard lattice as $2\theta_{K3}-3\epsilon$ where $\theta_{K3}$ is a square six class on the $K3$ surface $\hat{\Sigma}$. As already remarked in \Cref{sec:moduli}, up to an isometry (so after a change of the marking), we can suppose $\theta_{K3}$ to be $3u_1+u_2$ as in our setting. This proves \Cref{claim}.
	It is useful to describe also the transcendental lattice $\TR(\hat{\Sigma}^{[2]})$ of $\hat{\Sigma}^{[2]}$. So 
	\begin{equation*}
		\TR(\hat{\Sigma}^{[2]})\simeq \TR(\hat{\Sigma})\simeq U^{\oplus 2}\oplus E_8(-1)\oplus A_2(-1)\oplus D_4(-1).
	\end{equation*}
	In this case we cannot proceed as described in \Cref{DegSec}. Indeed, let us denote with $T_0=\eta(\Pic(\hat{\Sigma}^{[2]}))$ and $S_0=T_0^\perp=\eta(\TR(\hat{\Sigma}^{[2]}))$ and prove the following proposition.
	\begin{prop}
		There exists no order three isometry $\rho_0\in O(L)$ on the $K3^{[2]}$ lattice $L$ which extends the isometry $\text{id}_{T_0}\oplus\rho_{|S_0} \in O(T_0)\oplus O(S_0)$.
	\end{prop}
	\begin{proof}
		Suppose that it lifts to $\rho_0\in O(L)$, we want to to apply \Cref{teonik} to arrive to a contradiction. This means that there exists a subgroup $H< D_{S_0}$ which is glued to a subgroup $H'< D_{T_0}$. The only possibility is that $D_{D_4(-1)}<D_{S_0}$ is contained in the subgroup $H$. Moreover, when we restrict $\rho_0$ to $T_0$ it is the identity, so the order three isometry without fixed points $\rho_{|S_0}$ must be the identity on $H$. But we know from \Cref{IsomD4} that there exists only one, up to conjugation, order three isometry without fixed points on $D_4(-1)$ and it is not trivial on $D_{D_4(-1)}\simeq \left(\mathbb{Z}/2\mathbb{Z}\right)^2$. For the sake of completeness we write explicitly the order three isometry $\rho$ of $D_4(-1)$ without fixed points. \newline
		We can express this lattice as $D_4(-1)=\langle d_1, d_2, d_3, d_4\rangle$ with the following Gram matrix:\begin{equation*}
			\begin{pmatrix*}[r]
				
				-2 & 0 & -1 & 0\\
				0 & -2 & 1 & 0 \\
				-1 & 1 & -2 & 1 \\
				0 & 0 & 1 & -2
			\end{pmatrix*}
		\end{equation*}
		Then the isometry on the generators is given by:
		\begin{align*}
			d_1 &\mapsto -d_1+d_3-d_4 \\
			d_2 &\mapsto -d_1+d_2-d_3 \\
			d_3 &\mapsto -d_1-d_2+d_3+d_4 \\
			d_4 &\mapsto -d_2+d_3-d_4. 
		\end{align*}
		After a standard computation a basis of $D_{D_4(-1)}$ can be given by $\langle-d_1+\frac{1}{2}d_2+d_3+\frac{1}{2}d_4, \frac{1}{2}d_1-d_2-d_3-\frac{1}{2}d_4\rangle=\langle a, b\rangle$ and we see that $\rho^*(a)=b$ and $\rho^*(b)=a+b$, where $\rho^*$ is the map on $D_{D_4(-1)}$ induced by $\rho$.
	\end{proof}
	We still have a non-symplectic automorphism of order three on the $K3$ surface $\hat{\Sigma}$ and, thus, on $\hat{\Sigma}^{[2]}$. Hence, as noted in \Cref{rmk: roots}, we can see the latter as a point $(\hat{\Sigma}^{[2]}, \eta)$ having a degeneracy lattice $R_{\delta_1}\simeq R_{\delta_2}\simeq A_2(-1)\subset R_0$, ``forgetting'' one of the two roots orthogonal to its period (cfr. with \Cref{dfn:degenlattice}). We now call $T_{\delta_1}=\langle 6\rangle\oplus A_2(-1)\simeq \langle \theta, \delta_1\rangle$ and $S_{\delta_1}$ its orthogonal complement in $L$. Then it is well defined $\rho_{\delta_1}\in O(L)$, the lifting to $L$ of the isometry $\text{id}_{|T_{\delta_1}}\oplus\rho_{|S_{\delta_1}}$ to $L$. In order to see this, it is enough to note that in the proof of the lifting of the isometry in \Cref{thm:teoremino} the hypothesis on the Picard lattice is not used, the only thing used is the triviality of $\rho$ on the discriminant and in our case this holds on $R_{\delta_1}$. So, exactly as in the case $A_1$ done in \cite[$\S 4.3$]{BCS}, the isometry $\rho_{\delta_1}$ restricts to an isometry of $(\mathbb{Z}\epsilon)^{\perp}$. Note that the fixed part of $\rho_{\delta_1}$ is by definition $\langle 6\rangle\oplus A_2(-1)\simeq U(3)\oplus\langle -2\rangle$ and if we consider a primitive embedding $T_{\delta_1}\subset \Pic(\hat{\Sigma}^{[2]})$ we get $T_{\delta_1}\oplus A_2(-2)\simeq\langle 6\rangle\oplus A_2(-1)\oplus A_2(-2)\simeq U(3)\oplus A_2(-2)\oplus\langle -2\rangle\subset \Pic(\hat{\Sigma}^{[2]})$.
	\begin{lemma}
		The action of $\hat{\tau}$ on $H^2(\hat{\Sigma}, \mathbb{Z})$ is conjugate to the isometry $\rho_{\delta_1}$ restricted to $(\mathbb{Z}\epsilon)^{\perp}$. 
	\end{lemma}
	\begin{proof}
		The proof is a straight-forward application of \Cref{thm:namikawa}. Take $G=\langle \rho_{\delta_1} \rangle$, then condition $i)$ is simply verified. As $L^G\simeq U(3)$, also condition $iii')$ is easily verified. Lastly $S_{G,X}\simeq A_2(-2)$, thus also condition $ii)$ is verified, therefore there exists an automorphism $\bar{\tau}$ on $\hat{\Sigma}$ whose action is conjugate to $\rho_{\delta_1}$. \newline 
		If $\hat{\tau}=\bar{\tau}$ we are done. Otherwise, $\bar{\tau}$ is an order 3 automorphism whose invariant lattice is isomorphic to $U(3)$ and its fixed locus is a genus 4 curve $\bar{P}$. By \cite[Lemma 4.4, Proposition 4.7]{artebanisarti} the embedding $\Phi_{|\bar{P}|}: \hat{\Sigma}\to \mathbb{P}^4$ is such that we can choose coordinates $(x_1: \dots: x_5)$ of $\mathbb{P}^4$ such that the hyperplane $H$ whose preimage $\Phi^{-1}(H)=\bar{P}$ is given by $x_5=0$. Moreover, the induced automorphism on the image is the automorphism of $\mathbb{P}^4$ that maps $x_5\mapsto \zeta_3\cdot x_5$. In other words we can make a change of coordinates on $\hat{\Sigma}$ such that $\hat{\tau}=\bar{\tau}$. 
	\end{proof} 
	We now consider the natural automorphism $\hat{\tau}^{[2]}$ induced on $\hat{\Sigma}^{[2]}$ by $\hat{\tau}$. As its invariant lattice is $T_{\delta_1}$ we see that $(\hat{\Sigma}^{[2]}, \eta, \hat{\tau})$ defines a point in $\mathcal{N}_{T_{\delta_1}}^{\rho_{\delta_1}, \zeta}$. If $\hat{\Sigma}^{[2]}$ is moreover $K(T_{\delta_1})$-general we can conclude as in \Cref{sec:A1} and \Cref{sec:A3A4}; but, it turns out not to be the case as proved by the following proposition.
	\begin{prop}\label{prop:K-gen}
		$\hat{\Sigma}^{[2]}$ is not $K(T_{\delta_1})$-general.
	\end{prop}
	\begin{proof}
		Note that by \cite[Theorem 6.18]{Markman_surveyTorelli} the group of monodromies $\text{Mon}^2(\hat{\Sigma})$ acts transitively on the set of exceptional
		chambers, therefore, up to taking a
		different birational model, in order to prove that $\hat{\Sigma}^{[2]}$ is not $K(T_{\delta_1})$-general it is sufficient to show that there exists a wall divisor $\mu\in\Delta(\hat{\Sigma})$ not fixed by the action of $\hat{\tau}$ but such that $C^{\hat{\tau}}\cap \mu^{\perp}\neq 0$, where $C^{\hat{\tau}}$ denotes a connected component of the invariant positive cone. We will prove that there exist both $-2$ and $-10$ walls which are not fixed by $\hat{\tau}$ by exhibiting them. The wall divisor $E_1$ is not fixed by $\hat{\tau}$ but its orthogonal hyperplane intersects non-trivially $C^{\hat{\tau}}$, e.g in $2C+E_1+E_2+E_3$. The same is true for the wall divisor $2E_1+\epsilon$.
	\end{proof}
	This proposition implies that $(\hat{\Sigma}^{[2]}, \eta, \hat{\tau})$ is a non-separable point in $\mathcal{M}^{\rho_{\delta_1}, \zeta}_{T_{\delta_1}}$. Moreover this condition persists also if we consider $(\hat{\Sigma}^{[2]}, \eta, \hat{\tau})$ in the quotient $\mathcal{N}^{\rho_{\delta_1}, \zeta}_{T_{\delta_1}}$. Indeed, in \cite[Section 4]{BCS_complex} the authors show that the points in the fibre over a period $\omega\in\mathcal{N}^{\rho_{\delta_1}, \zeta}_{T_{\delta_1}}$ are in a one-to-one correspondence with the number of orbits of the monodromy action on the chambers. In the proof of \Cref{prop:K-gen} we found a $(-10)$-class defining a wall whose orbit cuts $K(T_{\delta_1})$, therefore there exists at least two chambers which are not in the same orbit of the monodromy action. This implies the following corollary.
	\begin{corol}
		There exist at least two non biregular models in $\mathcal{N}^{\rho_{\delta_1}, \zeta}_{T_{\delta_1}}$ which are over the same fiber of a period $\omega$ relative to a generic cubic threefold with a single singularity of type $A_2$.
	\end{corol}
	The pair $(\hat{\Sigma}^{[2]}, \eta)$ defines also a point in the moduli space of $(M,j)$-polarized IHS manifolds of $K3^{[2]}$-type, where $M:= U\oplus A_2(-2)\oplus\langle -2\rangle$. So we shift our interest in giving a formal characterization of marked IHS manifolds $(X, \phi)$ carrying the following commutative diagram in the data:
	\begin{equation}\label{diagram}
		\begin{tikzcd}
			T_{\delta_1}\ar[r, hook] \ar[rrd, hook, "j"] & M \ar[r, hook, "\iota"]  & \Pic(X)\subset H^2(X,\mathbb{Z}) \ar[d, "\phi"] \ar[r, "\sigma^*"] & H^2(X,\mathbb{Z}) \ar[d, "\phi"] \\
			&& L \ar[r, "\rho_{\delta_1}"] & L.
		\end{tikzcd}
	\end{equation}
    In order to arrive to a positive answer to our problem also for this case we want to study a general situation when an IHS manifold admits, like in this case, two compatible polarizations.
	\section{Kähler cone sections of $K$-type}\label{sec:conesections}
	In this section we introduce the notion of \emph{Kähler cone sections of $K$-type} in order to avoid the problems of separability illustrated in the previous chapter.\newline 
    
    Let $(X, \phi)$ be an $(M, j)$-polarized IHS manifold of type $L$ and let $T\subset M$ be a primitive sublattice. In this situation the pair $(X, \phi)$ is both $(M, j)$-polarized and $(T, j_{|T})$-polarized, so we can compare the notions of $K(M)$ and $K(T)$ generality when the two chambers are chosen in such a way that $\mathcal{K}_X\cap\iota(K(T))\neq\emptyset\neq\mathcal{K}_X\cap\iota(K(M))$, where $\iota$ denotes the $\mathbb{C}$-linear extension of the map $\iota$ of the diagram (\Cref{diagram}). Recall that $K(T)$ and $K(M)$ are, respectively, a connected component of $\mathcal{C}_T\setminus\bigcup_{\delta\in\Delta(S)}\delta^{\perp}$ and $\mathcal{C}_M\setminus\bigcup_{\delta\in\Delta(M)}\delta^{\perp}$. By their definition, $\Delta(T)=\Delta(M)\cap T$ and $\mathcal{C}_T= \mathcal{C}_M\cap (T\otimes \mathbb{R})$, therefore our choice of the chambers $K(M)$ and $K(T)$ implies $K(M)\cap (T\otimes \mathbb{R})\subset K(T)$ as the embedding of both via $\iota$ intersects $\mathcal{K}_X$.
	\begin{lemma} \label{lemma:Subsetgenerality}
		If a pair $(X, \phi)$ as above is $K(T)$-general then it is also $K(M)$-general.
	\end{lemma}
	\begin{proof}
		Suppose that $(X, \phi)$ is not $K(M)$-general: $\iota(\mathcal{C}_M)\cap\mathcal{K}_X$ is then a proper subset of $\iota(K(M))$. By \Cref{thm: AmerikVerbitsky} there exists $\lambda\in\Delta(X)$ such that $\lambda^{\perp}\cap\iota(K(M))\neq\emptyset$ and $\phi(\lambda)\notin\Delta(M)$. But remember that $\Delta(M)\supset\Delta(T)$ and by our choice of the chambers $K(M)\cap(T\otimes \mathbb{R})\subset K(T)$; then $\phi(\lambda)\notin\Delta(T)$ and $\lambda^{\perp}\cap\iota(K(T))\supset\lambda^{\perp}\cap\iota( K(M)\cap(T\otimes \mathbb{R}))\neq\emptyset$ implying that $(X, \phi)$ is not $K(T)$-general.
	\end{proof}
	\begin{rmk}
		The converse to the previous statement is not true. A counterexample is given by the IHS fourfold $\hat{\Sigma}^{[2]}$ in \Cref{sec: A2}. In fact, \Cref{prop:K-gen} shows that $(\hat{\Sigma}^{[2]}, \eta)$ is not $K(T_{\delta_1})$-general while it is clearly $K(M)$-general by \cite[Lemma 5.2]{BCS_complex}.
	\end{rmk}
	So, we  give the following more general definition.
	\begin{dfn}\label{def:Ktype}
		Let $K$ be a (connected and open) subset of a chamber $K(T)$ such that $K(T)\supset\phi(\mathcal{K}_X)$, with $\mathcal{K}_X$ denoting the Kähler cone of a $(T,j)$-polarized IHS manifold $(X, \phi)$. Then we say that $X$ has a Kähler cone section of $K$-type if $K=\phi(\mathcal{K}_X)\cap \mathcal{C}_T$. 
	\end{dfn}
	\begin{rmk}
		Note that if we choose a subset $K$ which is not connected or not open then no IHS manifold satisfies the above definition as its Kähler cone is connected and open.
	\end{rmk} 
	If the subset $K$ chosen in \Cref{def:Ktype} is a proper subset of $K(T)$ any IHS manifold $X$ with a Kähler section cone of $K$-type will not be $K(T)$-general. The following characterization gives us a link between the two definitions.
	\begin{prop}\label{prop:T-M polarizations}
		Let $(X,\phi)$ be a $(T,j_T)$-polarized IHS manifold of type $L$. If $X$ has a Kähler cone section of $K$-type then there exist:
		\begin{enumerate}[i)]
			\item a lattice $M$ with an embedding $j_M:M\hookrightarrow L$;
			\item a primitive embedding $\tilde{\iota}: T\hookrightarrow M$;
			\item a chamber $K(M)$, i.e. a connected component of $\mathcal{C}_M\setminus\bigcup_{\delta\in\Delta(M)}\delta^{\perp}$, with $K(M)\cap \mathcal{C}_T=K$;
		\end{enumerate}  such that $(X,\phi)$ is a $K(M)$-general, $(M, j_M)$-polarized IHS manifold.\newline Conversely, if $(X,\phi)$ is a $K(M)$-general, $(M, j_M)$-polarized IHS manifold with $T\subset M$ then it has a Kähler cone section of $K$-type with $K=K(M)\cap \mathcal{C}_T$. 
	\end{prop}
	\begin{proof}
		By \Cref{thm: AmerikVerbitsky} the Kähler cone of $X$ is a connected component of $\mathcal{C}_X\setminus\mathcal{H}_{\Delta}$, so denote with $\Lambda\subset\Delta(X)$ the set of those MBM classes $\lambda$ for which $\lambda^{\perp}$ is an extremal ray, i.e. if $\alpha,\beta\in \mathcal{C}_X$ are such that $\alpha+\beta\in \lambda^{\perp}$ then $\alpha, \beta\in\lambda^{\perp}$. Define $M$ as a lattice of minimal rank containing $T$ and $\phi(\Lambda)$ and such that every embedding in the chain of inclusions $T\subset M\subset M':=\phi(\Pic(X))$ is primitive and use the first to define $\tilde{\iota}$; fix an embedding $j_M:M\hookrightarrow L$ such that $(j_M)_{|T}=j_T$. Then by construction there exists a chamber $K(M)$ of $\mathcal{C}_M\setminus\mathcal{H}_{\Delta(M)}$ such that $(X, \phi)$ is $(M, j_M)$-polarized with $\iota_M:=j_M\circ\phi^{-1}$ and $K(M)$-general. Moreover, denoting with $\iota_T$ the embedding $T\hookrightarrow \Pic(X)$ we obtain $\iota_M (K(M)\cap \mathcal{C}_T)= \mathcal{K}_X\cap \iota_T(\mathcal{C}_T)= \iota_M(K)$ by hypothesis and injectivity of $\iota_M$, therefore the first part of the statement is proved. The second one is obvious using the fact that $\iota_T$ and $\iota_M$ are injective and that by construction $\iota_T=(\iota_M)_{|_T}$. Indeed $\iota_T(K)=\iota_T(K(M)\cap \mathcal{C}_T)=\iota_M(K(M))\cap\iota_T(\mathcal{C}_T)=\mathcal{K}_X\cap\iota_T(\mathcal{C}_T)$.
	\end{proof}
	\begin{rmk}
		Note that in the proof of the above proposition we chose to exhibit a minimal $M\supset T$ for which the statement holds. In fact, the same holds for every lattice $M'\supset M\supset T$ for which $(X, \phi)$ admits an $(M', j')$-polarization by \Cref{lemma:Subsetgenerality}. In particular it holds for $M'\simeq\Pic(X)$.
	\end{rmk}
	We now fix a chamber $K$ and $(X, \eta)$ a $(T, j_T)$-polarized IHS manifold with a Kähler cone section of $K$-type. According to \Cref{prop:T-M polarizations} we can find a pair $(M, j_M)$ such that $(X, \eta)$ is $(M, j_M)$-polarized, $K(M)$-general IHS manifold, with $K(M)\cap  \mathcal{C}_T=K$. We define $N:=j_M(M)^{\perp}$ and $S:= j_T(T)^{\perp}$. We can then consider the period map of the moduli space $\mathcal{M}_{K(M), j_M}$ of $(M, j_M)$-polarized IHS manifolds of type $L$ which are $K(M)$-general \begin{displaymath}
		\mathcal{P}_{K(M)}: \mathcal{M}_{K(M), j_M}\to \Omega^{M, j_M}\setminus\left(\mathcal{H}_{M} \cup \mathcal{H}'_{K(M)}\right)
	\end{displaymath} where a marked pair $(X, \phi)$ is sent to $\phi(H^{2,0}(X))$, which is an isomorphism by \cite[Theorem 3.13]{Camere_lattice_polarized_rmks}. \newline 
	We define the subfamily $\mathscr{F}_{j_M, K}^{T}$ of the $(T, j_T)$-polarized IHS manifold of type $L$ with the Kähler cone section of type $K$ and embedding $j_M$. Then we state the following theorem.
	\begin{thm}\label{thm:bijection}
		The period map $\mathcal{P}_{K(M)}$ restricted to $\mathscr{F}_{j_M, K}^{T}$ defines a bijection with $$\Omega_{j_M, K}^{T}:=\left\{\omega\in\mathbb{P}(N_\mathbb{C})\mid (\omega, \bar{\omega})>0, \, q(\omega)=0\right\}\setminus \left((\mathcal{H}_{S}\cap N_\mathbb{C}) \cup \mathcal{H}'_{K(M)}\right)$$ with $\mathcal{H}_{S_{\delta_1}}:=\cup_{\nu\in\Delta(S)}H_\nu$ and $\mathcal{H}'_{K(M)}:=\cup_{\nu\in\Delta'(K(M))}H_\nu$.
	\end{thm}
	\begin{proof}
		Consider the period map $\mathcal{P}_{K(M)}$. If we restrict this map to those IHS manifolds which are in  $\mathscr{F}_{j_M, K}^{T}$ then the image cannot lie in $\mathcal{H}_{S}\cap N_\mathbb{C}$ because if there existed $\nu^{\perp}\in \mathcal{H}_{S}$ such that $\omega:=\phi^{-1}(H^{2,0}(X))\in H_\nu$ we would have $\phi^{-1}(\nu)\in \NS(X)$ by definition of $\NS(X)$ as the orthogonal complement of $H^{2,0}(X)$ in $H^2(X,\mathbb{Z})$. Therefore $\eta^{-1}(\nu)$ would be a wall divisor yielding a contradiction using the same argument of \Cref{rmk:ample}.
		Being $\mathcal{P}_{K(M)}$ an injective morphism, in order to get a bijection we just need to show what is the image. Take a period $\omega\in\Omega_{j_M, K}^{T}$ and consider $(X,\phi)=\mathcal{P}_{K(M)}^{-1}(\omega)$ then it is immediate to see that $(X,\phi)$ is indeed an element of $\mathscr{F}_{j_M, K}^{T}$ as it is $K(M)$-general and by \Cref{prop:T-M polarizations} it has a Kähler cone section of $K$-type.
	\end{proof}
	\section{A moduli space for the singularity of type $A_2$}\label{sec:A2works}
	In this section we give the proof of \Cref{thm:main} for $\Delta_3^{A_2}$ and continue the description started at the end of \Cref{sec: A2}, therefore we will use the same notation resumed in the diagram \Cref{diagram}. Moreover, we recall that $T_{\delta_1}= U(3)\oplus\langle -2\rangle$ and $M= U\oplus A_2(-2)\oplus\langle -2\rangle$.\newline

	Let us define $K(M)$ as the connected component of $\mathcal{C}_M\setminus\cup_{\nu\in\Delta(M)}H_\nu$ which contains a Kähler class of $\hat{\Sigma}^{[2]}$, i.e. $\mathcal{K}_{\hat{\Sigma}^{[2]}}\subset \iota(K(M))\otimes \mathbb{R}$. We choose in a compatible way also $K(T_{\delta_1})$, i.e. $\mathcal{K}_{\hat{\Sigma}^{[2]}}\cap\iota(K(T_{\delta_1}))\neq\emptyset\neq\mathcal{K}_{\hat{\Sigma}^{[2]}}\cap\iota(K(M))$; we also choose $K:=K(M)\cap (T_{\delta_1}\otimes\mathbb{R})\subset K(T_{\delta_1})$. Moreover, in the case of $\hat{\Sigma}^{[2]}$, the map $i:M\to\Pic(\hat{\Sigma}^{[2]})$ is an isomorphism and therefore $i(K(M))=\mathcal{K}_{\hat{\Sigma}^{[2]}}$. So, we define the family $\mathscr{F}_{K, T_{\delta_1}}^{\rho_{\delta_1}, \zeta}\subset\mathcal{M}_{T_{\delta_1}}^{\rho_{\delta_1}, \zeta}$ of IHS manifolds in $\mathcal{M}_{T_{\delta_1}}^{\rho_{\delta_1}, \zeta}$ which have a Kähler cone section of $K$-type and admit an embedding $T_{\delta_1}\hookrightarrow M\hookrightarrow\Pic(X)$ compatible with the polarization, i.e. for which the commutative diagram (\Cref{diagram}) is defined.
	\begin{rmk}\label{remarkone Stevell}
		We do not want to choose the embedding of $M$ in $L$ in the data but only the embedding $j$ of $T_{\delta_1}$. It is an easy exercise using \cite[Proposition 1.15.1]{Nik1980} to see that there are only two possible non-isomorphic embeddings of $T_{\delta_1}$ in $L$ and they correspond respectively to the two possible non-isomorphic embeddings of $M$ in $L$ (see \Cref{esercizio}). Recall that we say that two embeddings $j_1$ and $j_2$ of $M$ in $L$ are isomorphic if there exists an automorphism $\varphi\in O(L)$ such that $j_2=\varphi\circ j_1$. Therefore we choose the embedding $j_k$ such that $j_k(T_{\delta_1})^\perp=j(T_{\delta_1})^\perp=: S_{\delta_1}\simeq U(3)\oplus U\oplus E_8(-1)^{\oplus 2}\supset \TR(\hat{\Sigma}^{[2]})$. Note that doing so $(\hat{\Sigma}^{[2]}, \eta)\in\mathscr{F}_{K, T_{\delta_1}}^{\rho_{\delta_1}, \zeta}$. 
	\end{rmk}
	\begin{prop}\label{prop:bijection}
		The period map $\mathcal{P}_{K(M)}$ restricted to $\mathscr{F}_{K, T_{\delta_1}}^{\rho_{\delta_1}, \zeta}$ defines a bijection with $$\Omega_{K, T_{\delta_1}}^{\rho_{\delta_1}, \zeta}:=\left\{\omega\in\mathbb{P}(N_\mathbb{C}(\zeta))\mid (\omega, \bar{\omega})>0\right\}\setminus \left((\mathcal{H}_{S_{\delta_1}}\cap N_\mathbb{C}) \cup \mathcal{H}'_{K(M)}\right)$$ with $\mathcal{H}_{S_{\delta_1}}:=\cup_{\nu\in\Delta(S_{\delta_1})}H_\nu$ and $\mathcal{H}'_{K(M)}:=\cup_{\nu\in\Delta'(K(M))}H_\nu$.
	\end{prop}
	\begin{proof}
		This proposition states that we can apply \Cref{thm:bijection} also with the additional structure of the $(\rho_{\delta_1},T_{\delta_1})$-polarization. The image of $\mathcal{P}_{K(M)}$ lies in the eigenspace $\mathbb{P}(N_\mathbb{C}(\zeta))$ by definition of Picard group. 
		Take now a period $\omega\in\Omega_{M,K(M)}^{\rho_{\delta_1}, \zeta}$ and consider $(X,\phi)=\mathcal{P}_{K(M)}^{-1}(\omega)$, we want to show that on $X$ there exists an automorphism $\sigma$ satisfying the properties required in diagram (\Cref{diagram}). Define $\sigma^*:=\phi^{-1}\circ\rho_{\delta_1}\circ\phi$. It is an isomorphism of integral Hodge structures since $$\sigma^*(\omega_X):=\sigma^*(\mathcal{P}_{K(M)}^{-1}(\omega))=\phi^{-1}(\rho_{\delta_1}(\omega))=\zeta\omega_X.$$ Moreover, it is a parallel transport operator as $\rho_{\delta_1}\in\Mon^2(L)$. If it preserves also a Kähler class we can conclude with the Hodge theoretic Torelli \Cref{thm: hodge theoretic}. But $X$ has a Kähler cone section of $K$-type, therefore $$\mathcal{K}_X\cap i(T_{\delta_1})\supset \mathcal{K}_X\cap i(T_{\delta_1})\cap i(\mathcal{C}_M)= i(K(M))\cap i(T_{\delta_1})\neq \emptyset$$ thus, $\rho_{\delta_1}$ fixes a Kähler class. 
	\end{proof}
	
	\begin{rmk}
		Note that if $(X_1, \phi_1)$ and $(X_2, \phi_2)$ define the same point in $\mathcal{M}_{M, K(M)}$, i.e. there exists a biregular morphism $f: X_1\to X_2$ such that $\phi_1=\phi_2\circ f^*$ and $i_1=i_2\circ f^*$, then also $\sigma_1=f^{-1}\circ\sigma_2\circ f$ by \cite[Theorem 1.8]{brandhorstCattaneo}.
	\end{rmk}
	Let\begin{equation*}
		\begin{split}
			\Mon^2(M,j,\rho_{\delta_1}):=&\left\{g\in\Mon^2(L)\mid g(M)=M, \, g(t)=t,\right. \\ &\left.g\circ\rho_{\delta_1}=\rho_{\delta_1}\circ g, \, \forall t\in T_{\delta_1}\right\}
		\end{split}
	\end{equation*}  and denote its image in $O(N)$ via the restriction map with $\Gamma^{\rho_{\delta_1}}_{M, j}$. Note that $\Mon^2(M,j,\rho_{\delta_1})$ is the stabilizer of $\mathscr{F}_{K, T_{\delta_1}}^{\rho_{\delta_1}, \zeta}$ for the action of $\Mon^2(M, j)$ on $\mathcal{M}_{M, K(M, T_{\delta_1})}$. Moreover, by definition, the bijection defined in \Cref{prop:bijection} is equivariant with respect to the action of $\Mon^2(M,j,\rho_{\delta_1})$ and of $\Gamma^{\rho_{\delta_1}}_{M, j}$. Denoting $\mathscr{N}_{K, T_{\delta_1}}^{\rho_{\delta_1}, \zeta}:=\mathscr{F}_{K, T_{\delta_1}}^{\rho_{\delta_1}, \zeta}/\Mon^2(M,j,\rho_{\delta_1})$ we deduce the following corollary:
	\begin{corol}\label{corol:bijection}
		There exists a bijection between $\mathscr{N}_{K, T_{\delta_1}}^{\rho_{\delta_1}, \zeta}$ and $\Omega_{K, T_{\delta_1}}^{\rho_{\delta_1}, \zeta}/\Gamma^{\rho_{\delta_1}}_{M, j}.$
	\end{corol}
	The whole formal construction made above is ``natural'' in the sense that it arises from the theory of $K3$ surfaces in a compatible way. Let us clarify this sentence. Consider a diagram similar to (\Cref{diagram}) where $(X, \phi)$ is, this time, a marked $K3$ surface such that the following commutative diagram is defined
	\begin{equation*}
		\begin{tikzcd}
			U(3)\ar[r, hook] \ar[rrd, hook, "j"] & U\oplus A_2(-2) \ar[r, hook, "i"]  & \Pic(X)\subset H^2(X,\mathbb{Z}) \ar[d, "\phi"] \ar[r, "\sigma^*"] & H^2(X,\mathbb{Z}) \ar[d, "\phi"] \\
			&& L_{K3} \ar[r, "\tilde{\rho}"] & L_{K3}.
		\end{tikzcd}
	\end{equation*}
	Where $\tilde{\rho}$ is $\rho_{\delta_1}$ restricted to $\langle -2\rangle^{\perp}$. Define the subfamily of the $K3$ surfaces with an ample $U(3)$-polarization for which there exists the above diagram as $\mathscr{K}^{\tilde{\rho}, \zeta}_{U(3), U\oplus A_2(-2)}$ and note that $\hat{\Sigma}\in\mathscr{K}^{\tilde{\rho}, \zeta}_{U(3), U\oplus A_2(-2)}$. Then with the same techniques of \Cref{prop:bijection} we can find a generalized version of \cite[Theorem 11.3]{DolgachevKondo} which applies to our subfamily and obtain the following proposition. We denote with \begin{equation*}
		\begin{split}
			\Gamma_{U(3), U\oplus A_2(-2)}^{\tilde{\rho}}:=\{\gamma\in O(L_{K3})\mid \  &\gamma_{|U\oplus A_2(-2)}\in O(U\oplus A_2(-2)), \\ &\gamma_{|U(3)}=\id \ \text{and} \ \gamma\circ\tilde{\rho}=\tilde{\rho}\circ\gamma\}.
		\end{split}
	\end{equation*}
	\begin{prop} \label{prop:K3A2}
		The period map defines a bijection between $\mathscr{K}^{\tilde{\rho}, \zeta}_{U(3), U\oplus A_2(-2)}$ and $$\Omega_{U(3), U\oplus A_2(-2)}^{\tilde{\rho}, \zeta}:=\left\{\omega\in\mathbb{P}(N_\mathbb{C}(\zeta))\mid (\omega, \bar{\omega})>0\right\}\setminus \left(\mathcal{H}_{S_{\delta_1}}\cap N_\mathbb{C}\right)$$ with $\mathcal{H}_{S_{\delta_1}}:=\cup_{\nu\in\Delta(S_{\delta_1})}H_\nu$. Moreover, this bijection descends to bijection at the level of isomorphism classes, i.e. $\Omega_{U(3), U\oplus A_2(-2)}^{\tilde{\rho}, \zeta}/ \Gamma_{U(3), U\oplus A_2(-2)}^{\tilde{\rho}}$ parametrizes isomorphism classes of $K3$ surfaces in $\mathscr{K}^{\tilde{\rho}, \zeta}_{U(3), U\oplus A_2(-2)}.$
	\end{prop}
	\begin{proof}
		Consider the period map for the $K3$ surfaces which are $(\tilde{\rho}, U(3))$-ample, by \cite[Theorem 11.2]{DolgachevKondo} it is a bijection with $$\Omega_{U(3)}^{\tilde{\rho}}:=\left\{\omega\in\mathbb{P}((S_{\delta_1}\otimes\mathbb{C})(\zeta))\mid (\omega, \bar{\omega})>0\right\}\setminus\mathcal{H}_{S_{\delta_1}}.$$ If we moreover consider the restriction to our subfamily then it is easy to see that \cite[Theorem 10.2]{DolgachevKondo} guarantees the assertion. 
	\end{proof}
	We can now state the following proposition which answers our question in a formal way. \begin{prop} Let $R_0=\Span(\delta_1, \delta_2, \rho(\delta_1), \rho(\delta_2))\simeq D_4(-1)$ be the degeneracy lattice relative to a period $\omega$ of a generic cubic threefold having a single singularity of type $A_2$. Then the $A_2$ locus $\Delta^{A_2}_3$ is birational to the moduli space $\mathscr{N}_{K, T_{\delta_1}}^{\rho_{\delta_1}, \zeta}$. \end{prop} 
	\begin{proof}
		The proof has again the same structure of the proof of \Cref{prop:A1}, we review here the main points. Recall that the extension of the period map $\mathcal{P}^3:\mathcal{C}^{\text{sm}}_3\to\frac{\mathbb{B}^{10}\setminus\left(\mathcal{H}_n\cup \mathcal{H}_c\right)}{\mathbb{P}\Gamma}$ to the nodal locus is done by \cite{ACT} defining its period as the period of its associated $K3$ surface (see also \Cref{sec: riassunto ACT}). In our case the generic $A_2$ nodal cubic has the period 
		\begin{equation*}
			\mathcal{P}^3(C):= \mathcal{P}^{\tilde{\rho}, \zeta}_{U(3), U\oplus A_2(-2)}((\Sigma, \eta)).
		\end{equation*}
		The latter period map is the same of \Cref{prop:K3A2}. Following \Cref{prop:K3A2} this map yields an isomorphism between $\mathscr{K}^{\tilde{\rho}, \zeta}_{U(3), U\oplus A_2(-2)}$ and $\Omega_{U(3), U\oplus A_2(-2)}^{\tilde{\rho}, \zeta}$.
		
		As proven in \Cref{sec: A2} $(\hat{\Sigma}^{[2]}, \eta, \tau)$ defines a point in $\mathscr{F}_{K, T_{\delta_1}}^{\rho_{\delta_1}, \zeta}$. The period map in this space defines a bijection which, according to \Cref{corol:bijection}, descends to a bijection between $\mathscr{N}_{K, T_{\delta_1}}^{\rho_{\delta_1}, \zeta}$ and $\Omega_{K, T_{\delta_1}}^{\rho_{\delta_1}, \zeta}/\Gamma^{\rho_{\delta_1}}_{M, j}$.
		Note that $\Omega_{K, T_{\delta_1}}^{\rho_{\delta_1}, \zeta}/\Gamma^{\rho_{\delta_1}}_{M, j}$ and $$\Omega_{U(3), U\oplus A_2(-2)}^{\tilde{\rho}, \zeta}/ \Gamma_{U(3), U\oplus A_2(-2)}^{\tilde{\rho}}$$ are birational. So, in order to conclude the proof, it is enough to show that $\Delta^{A_2}_3$ is birational to the isomorphism classes of $K3$ surfaces in $\mathscr{K}^{\tilde{\rho}, \zeta}_{U(3), U\oplus A_2(-2)}$ as the claim of the proposition would follow through a composition of birational morphisms. But this is true as in \cite[Section 6]{ACT} the authors show that the discriminant locus maps isomorphically to its image, the nodal hyperplane arrangement, through the period map, therefore a generic point in $\Delta^{A_2}_3$ is mapped isomorphically to the period of a generic $K3$ surface in $\mathscr{K}^{\tilde{\rho}, \zeta}_{U(3), U\oplus A_2(-2)}$ and by \Cref{prop:K3A2} this is an isomorphism.
	\end{proof}
	
	\begin{rmk}
		Also Kond{\={o}} notes in \cite{Kondo} this family as a codimension 1 family in the moduli space of curves of genus 4 which consists of smooth curves with a vanishing theta null.
	\end{rmk}

\section*{Acknowledgments}

I would like to thank Samuel Boissière, Chiara Camere and Alessandra Sarti for having introduced me to this beautiful subject. This work is part of my PhD thesis and I would like to thank again my advisors, Chiara Camere and Alessandra Sarti, for their many hints, suggestions and constant support throughout the redaction. I would like to thank also Enrico Fatighenti for having read a preliminary version of this paper and the anonymous referee for the suggestions on how to improve the readability of this paper. I would like to thank a lot Cédric Bonnafé for having suggested and showed me how to use Springer theory to prove the existence of only one isometry, up to conjugation, of order three on $D_4$ and $E_6$ as done in \Cref{sec:ComplexReflection}. Moreover, I would like to thank Alice Garbagnati, Stevell Muller and Benedetta Piroddi for the helpful discussions. Finally, I would like to thank the Departments of Mathematics of ``Université de Poitiers'' and ``Università degli Studi di Milano'' for the fundings; I am also partially funded by the grant ``VINCI 2021-C2-104'' issued by the ``Université Franco Italienne''.

\appendix
 \section{Complex Reflection Groups and Springer Theory} \label{sec:ComplexReflection}
	In this section we will recall some notions about reflection groups and Springer theory. A deeper reference for this is \cite{broue} or \cite[Section 3]{bonnafe_sarti}.\newline

	Let $V$ be a complex vector space of $\dim_{\mathbb{C}}(V)=n$ and $W\subset GL_n(V)$ a finite subgroup of the complex general linear group of degree $n$. Given an element $g\in GL_n(V)$ we define the subset $V^g:=\left\{v\in V\mid g(v)=v\right\}\subset V$ of the elements fixed by $g$. Moreover we define also the set\begin{equation*}
		\text{Ref}(W):=\left\{s\in W\mid \dim(V^s)=n-1\right\}.
	\end{equation*}
	\begin{dfn}
		A finite subgroup $W\subset GL_n(V)$ is called a \emph{complex reflection group} if $W=\langle\text{Ref}(W)\rangle$.
	\end{dfn}
	A classic result, stated e.g. in \cite[Theorem 4.1]{broue}, is the following
	\begin{thm*}[Serre--Chevalley, Shepherd--Todd]
		Given a complex reflection group $W$ acting on a complex vector space $V$ of $\dim_{\mathbb{C}}(V)=n$ then there exist $f_1,\dots, f_n$ homogeneous polynomials of degree $d_1, \dots, d_n$ such that the invariant ring by the action of $W$ is given by
		\begin{equation*}
			\mathbb{C}[V]^W=\mathbb{C}[f_1, \dots, f_n].
		\end{equation*}
	\end{thm*}
	The family $\left\{f_1, \dots, f_n\right\}$ is not unique (up to permutations) but the degrees $\left\{d_1, \dots, d_n\right\}$ are uniquely determined (up to permutations) by $V$ and $W$. Moreover, a result by Solomon \cite[Theorem 4.44 and Section 4.5.4]{broue} implies that the graded $\mathbb{C}[V]^W$-module of $W$-invariant derivations of $\mathbb{C}[V]$ admits a homogeneous $\mathbb{C}[V]^W$-basis $(g_1, \dots, g_n)$ whose degrees $(d^*_1, \dots, d^*_n)$ are called co-degrees. Also the co-degrees are invariant up to permutation. \newline
	Now, in order to state the results that we need from Springer theory, we need to define the following numbers for any $e\in\mathbb{N}$
	\begin{equation*}
		\lambda(e):=|\left\{1\leq i\leq n\mid e \ \text{divides} \  d_i\right\}|
	\end{equation*}
	\begin{equation*}
		\lambda^*(e):=|\left\{1\leq i\leq n\mid e \ \text{divides} \  d^*_i\right\}|.
	\end{equation*}
	Moreover, if the primitive $e$-th root of unity $\zeta_e$ is a eigenvalue for the action of $w\in W$, we set $V(w, \zeta_e)$ to be the eigenspace of $V$ with respect to the action of $w$ relative to $\zeta_e$. 
	Then, putting together various results from \cite[Theorem C]{LehrerSpringer+2011+181+194} and \cite[Theorem 3.4, Theorem 4.2, Theorem 6.2]{Springer1974}, we state the following
	\begin{thm}[Springer, Lehrer-Springer] \label{thm:Springer}
		Let $W$ be a complex reflection group acting on $V$. Then for every $e\in\mathbb{N}$ it holds $\lambda(e)=\max_{w\in W} \dim(V(w, \zeta_e))$. If, moreover, $e$ is such that $\lambda(e)=\lambda^*(e)$, then the elements $w_e$ which attain the maximum define a unique conjugacy class in $W$.
	\end{thm}
	Let us explain the results in this section with a couple of examples.
	\begin{ex}[Isometries of order 3 on $D_4$]\label{IsomD4}
	From classical theory, that can be found e.g. in \cite[Chapter 4, Section 7]{ConwaySloane}, the group of lattice isometries of $D_4$ is $O(D_4)\simeq G_{28}\simeq W(F_4)=:W$ where $W(F_4)$ denotes the Coxeter group $F_4$ and $G_{28}$ is the 28-th group in the Shepherd--Todd classification. If we consider its action on $V=\mathbb{C}^4$ it is a complex reflection group. In order to compute the degrees and co-degrees one can use a computer algebra program like MAGMA or refer to \cite{bonnafe2018singular}. In any case we have that the degrees are \begin{equation*}
		(d_1, d_2, d_3, d_4)=(2, 6, 8, 12)
	\end{equation*} and the co-degrees are \begin{equation*}
		(d_1^*, d_2^*, d_3^*, d_4^*)=(0, 4, 6, 10).
	\end{equation*}
	So we can look at $\lambda(3)=\max_{w\in W} \dim (V(w, \zeta_3))=2=\lambda^*(3)$. Suppose that this maximum is attained in $w_3\in W= W(F_4)\subset GL_4(V)$. Therefore $w_3$ is an integer matrix which admits $\zeta_3$ as eigenvalue whose eigenspace is 2-dimensional, thus the same is true also for $\bar{\zeta_3}$. Moreover, as the triple $(W, V, 3)$ satisfies the hypothesis of \Cref{thm:Springer}, $w_3$ is unique up to conjugation. Wrapping up everything, we proved that there exists only one isometry of $D_4$, up to conjugation, of order three and without fixed points (the last statement comes from the fact that the only eigenvalues are $\zeta_3$ and $\bar{\zeta_3}$).
    \end{ex}
	\begin{ex}[Isometries of order 3 on $E_6$]\label{isomE_6}
	Here we want to apply again the same idea. Lattice isometries of $E_6$ are just $O(E_6)\simeq \mathbb{Z}/2\mathbb{Z}\times G_{35}\simeq \mathbb{Z}/2\mathbb{Z}\times W(E_6)$ (again it is a classical result which can be found on \cite[Chapter 4, Section 8.3]{ConwaySloane}) with the same convention as above. $W:=W(E_6)$ acts as a complex reflection group on $V=\mathbb{C}^6$. The degrees are
	\begin{equation*}
		(d_1, d_2, d_3, d_4, d_5, d_6)=(2, 5, 6, 8, 9, 12)
	\end{equation*}
	and the co-degrees are
	\begin{equation*}
		(d_1^*, d_2^*, d_3^*, d_4^*, d_5^*, d_6^*)=(0, 3, 4, 6, 7, 10).
	\end{equation*}
	Now $\lambda(3)=\max_{w\in W} \dim (V(w, \zeta_3))=3=\lambda^*(3)$ and we assume that $w_3$ attains the maximum. With the same argument as above we obtain that $w_3$ has $\zeta_3$ and $\bar{\zeta}_3$ as triple eigenvalues. Therefore we proved that up to conjugation (and sign) there exists only one order three isometry without fixed points on $E_6$.
    \end{ex}
 \section{Dimension of moduli spaces}\label{appendix} 
	In this section we compute the dimensions of the moduli spaces of families which are complete intersections in $\mathbb{P}^4$ of a quadric hypersurface of rank 3 or 4 and a cubic threefold.\newline 
    
    Using the generalized Morse lemma and the Recognition principle as done e.g. in \cite{Heckel} one can arrive to a generic form for the families we are interested in and count the free parameters. But the computations are long so we will use the following result which is a direct application of the Generalised Morse Lemma (see \cite[I, Theorem 2.47]{GLS} for a possible reference) and the Recognition Principle \cite[Lemma 1]{BruceWall}. As this theorem appears on a PhD dissertation which has not been published at the day we are writing this article we include here its proof.
	\begin{thm}[\protect{\cite[Theorem~1.15]{Heckel}}]\label{thmHeckel}
		Let $Y\subset \mathbb{A}_{\mathbb{C}}^n$ be a hypersurface defined by a polynomial $P\in \mathbb{C}[x_1, \dots, x_n]$
		and assume that the origin is an isolated singular point of $Y$ of corank one. Then, there exist polynomials $C_1, \dots, C_{k+1}$ in the coefficients of $P$ and depending on the choice of an analytic coordinate change such that the conditions\begin{equation*}
			C_1=\dots=C_k=0, \, C_{k+1}\neq 0
		\end{equation*}
		on the coefficients of $P$ are equivalent to $(Y, 0)$ being of type $A_k$. Moreover, each $C_i$ is homogeneous of degree $i-2$ and fixing the analytic coordinate change they depend on, there is an explicit algorithm computing them.
	\end{thm}
	\begin{proof}
		Let $k\in\mathbb{N}$. Using the generalized Morse lemma we suppose that, after a suitable analytic coordinates change, $P$ has the form:
		\begin{equation*}
			P(x)=x_1^2+\dots +x_{n-1}^2+P_3(x_n)+\dots +P_{k+1}(x_n)+\sum_{i=1}^{n-1}x_iQ_i(x_1, \dots, x_n)
		\end{equation*}
		where each $P_i$ is a polynomial of degree $i$ and each $Q_i$ of degree $k$. In order to apply the recognition Principle we take the weight $\alpha(A_k)=\left(\frac{1}{2}, \dots, \frac{1}{2}, \frac{1}{k+1}\right)$ and note that the terms of degree $\alpha(A_k)<1$ are $P_3(x_n)+\dots +P_{k}(x_n)$, the terms of degree $\alpha(A_k)=1$ are $x_1^2+\dots +x_{n-1}^2+P_{k+1}(x_n)$ and the terms of degree $\alpha(A_k)>1$ are $\sum_{i=1}^{n-1}x_iQ_i(x_1, \dots, x_n)$. Therefore we write $C_ix_n=P_i(x_n)$ and conclude using the recognition Principle.
	\end{proof}
	This will lead us to prove the following proposition.
	\begin{prop}
		The dimension of the family $\mathcal{K}_{A_i}$ associated to the cubic threefold having one $A_i$ singularity (and thus of the subfamily of cubic threefolds having an isolated singularity of type $A_i$) for $i=1, \dots, 4$ is $10-i$. A generic element $\Sigma_{A_i}$ in $\mathcal{K}_{A_i}$ is such that $\text{rk}(\Pic(\Sigma_{A_i}))=2i$
	\end{prop}
	\begin{proof}
		We start from the most general case which is the $A_1$ case. Note that this is the only corank 0 case so the theorem does not apply in this case. The equations are given by \begin{equation}
			\begin{cases}
				f_2(x_1,x_2,x_3,x_4)=0 \\
				f_3(x_1,x_2,x_3,x_4)+ax_5^3=0.
			\end{cases}
		\end{equation}
		So we have $\binom{3+2}{2}=10$ parameters for the quadric and $\binom{3+3}{3}+1=21$ for the cubic. Then we have to impose $4$ conditions because if two cubic hypersurfaces differ by a multiple of the quadric they yield the same intersection. As every equation is defined up to a constant the parameters are $10+21-4-1-1=25$. Then we have to consider the projective transformations which preserve the family, as projectivities are up to a constant are $4\cdot4+1-1=16$. Finally, the dimension of this family is $25-16=9$. Now, we consider the family $\mathcal{K}_{A_i}$ associated to the cubic threefold having one $A_i$ singularity for $i\geq 2$. In these cases we have the same parameters and projective transformations as before but we need to add $1$ condition for being a corank 1 singularity (this is equivalent to ask that $f_2=0$ has rank 3 as a quadric) and $i-2$ conditions coming from \Cref{thmHeckel}. Therefore the dimension of the moduli space of the family of $(2,3)$-complete intersections in $\mathbb{P}^4$ associated to a cubic threefold having a singularity of type $A_i$ is $10-i$. Then if we take a generic element $\Sigma_{A_i}$ in $\mathcal{K}_{A_i}$ then by \cite[Section 9]{ArtebaniSartiTaki} we obtain $\text{rk}(\Pic(\Sigma_{A_i}))=22-2(10-i+1)=2i$.
	\end{proof}	
 \section{An easy exercise}\label{esercizio}
 Here we outline the execution of the exercise mentioned in \Cref{remarkone Stevell}.\newline

 Let $L$ be the $K3^{[2]}$ lattice. Then $D_L\simeq \mathbb{Z}/2\mathbb{Z}$ with finite quadratic form $q_L=\langle \frac{3}{2}\rangle$. Moreover let $T=U(3)\oplus\langle -2\rangle$ and $M=U\oplus A_2(2)\oplus\langle -2\rangle$. Clearly, $q_T\simeq q_L\oplus q_{U(3)}$ and $q_M\simeq q_L\oplus q_{A_2(2)}$, so given \cite[Proposition 1.15.1]{Nik1980}, recalled in \Cref{teonik}, the only possibilities for the respective orthogonal complements for $T$ in $L$ are the following:
 \begin{itemize}
     \item the genus of the lattice with signature $(2, 18)$ and discriminant form $q_T(-1)\oplus q_L$ is non-empty. Using the notation of Conway--Sloane (\cite{ConwaySloane}) this is $II_{(2, 18)}2_I^{+2}3^{-2}$. There exists only one class of isomorphism represented by \begin{equation*}
         U\oplus U(3)\oplus E_7\oplus E_8 \oplus \langle -2\rangle.
     \end{equation*}
     \item the genus of the lattice with signature $(2, 18)$ and discriminant form $q_{U(3)}(-1)$ is non-empty. Using the notation of Conway--Sloane (\cite{ConwaySloane}) this is $II_{(2, 18)}3^{-2}$. There exists only one class of isomorphism represented by \begin{equation*}
         U\oplus U(3)\oplus E_8^{\oplus 2}.
     \end{equation*}
 \end{itemize}
 Analogously for $M$:
 \begin{itemize}
     \item the genus of the lattice with signature $(2, 16)$ and discriminant form $q_M(-1)\oplus q_L$ is non-empty. Using the notation of Conway--Sloane (\cite{ConwaySloane}) this is $II_{(2, 16)}2_I^{-4}3^{+1}$. There exists only one class of isomorphism represented by \begin{equation*}
         U^{\oplus 2}\oplus E_8 \oplus D_4\oplus \langle -6\rangle\oplus \langle -2\rangle.
     \end{equation*}
     \item the genus of the lattice with signature $(2, 16)$ and discriminant form $q_{A_2(2)}(-1)$ is non-empty. Using the notation of Conway--Sloane (\cite{ConwaySloane}) this is $II_{(2, 16)}2_{II}^{-2}3^{+1}$. There exists only one class of isomorphism represented by \begin{equation*}
         U^{\oplus 2}\oplus E_8 \oplus D_4\oplus A_2.
     \end{equation*}
 \end{itemize}

\bibliographystyle{alpha}
\bibliography{bibliographydoct}

\begin{thebibliography}{CMGHL23}

\bibitem[ACT11]{ACT}
Daniel Allcock, James~A. Carlson, and Domingo Toledo.
\newblock The moduli space of cubic threefolds as a ball quotient.
\newblock {\em Mem. Amer. Math. Soc.}, 209(985):xii+70, 2011.

\bibitem[All03]{AllcockSing}
Daniel Allcock.
\newblock The moduli space of cubic threefolds.
\newblock {\em J. Algebraic Geom.}, 12(2):201--223, 2003.

\bibitem[AS08]{artebanisarti}
Michela Artebani and Alessandra Sarti.
\newblock Non-symplectic automorphisms of order 3 on {$K3$} surfaces.
\newblock {\em Math. Ann.}, 342(4):903--921, 2008.

\bibitem[AST11]{ArtebaniSartiTaki}
Michela Artebani, Alessandra Sarti, and Shingo Taki.
\newblock {$K3$} surfaces with non-symplectic automorphisms of prime order.
\newblock {\em Math. Z.}, 268(1-2):507--533, 2011.
\newblock With an appendix by Shigeyuki Kond{\={o}}.

\bibitem[AV15]{amerikverbitsky2014}
Ekaterina Amerik and Misha Verbitsky.
\newblock Rational curves on hyperk\"{a}hler manifolds.
\newblock {\em Int. Math. Res. Not. IMRN}, (23):13009--13045, 2015.

\bibitem[BC22]{brandhorstCattaneo}
Simon Brandhorst and Alberto Cattaneo.
\newblock {Prime Order Isometries of Unimodular Lattices and Automorphisms of
  IHS Manifolds}.
\newblock {\em International Mathematics Research Notices}, 10 2022.
\newblock rnac279.

\bibitem[BCS19a]{BCS_complex}
Samuel Boissi\`ere, Chiara Camere, and Alessandra Sarti.
\newblock Complex ball quotients from manifolds of {$K3^{[n]}$}-type.
\newblock {\em J. Pure Appl. Algebra}, 223(3):1123--1138, 2019.

\bibitem[BCS19b]{BCS}
Samuel Boissi\`ere, Chiara Camere, and Alessandra Sarti.
\newblock Cubic threefolds and hyperk\"{a}hler manifolds uniformized by the
  10-dimensional complex ball.
\newblock {\em Math. Ann.}, 373(3-4):1429--1455, 2019.

\bibitem[BD85]{BD}
Arnaud Beauville and Ron Donagi.
\newblock La vari\'{e}t\'{e} des droites d'une hypersurface cubique de
  dimension {$4$}.
\newblock {\em C. R. Acad. Sci. Paris S\'{e}r. I Math.}, 301(14):703--706,
  1985.

\bibitem[Bea83a]{Beauvilleremarks}
Arnaud Beauville.
\newblock Some remarks on {K}\"{a}hler manifolds with {$c_{1}=0$}.
\newblock In {\em Classification of algebraic and analytic manifolds ({K}atata,
  1982)}, volume~39 of {\em Progr. Math.}, pages 1--26. Birkh\"{a}user Boston,
  Boston, MA, 1983.

\bibitem[Bea83b]{Beauville1983}
Arnaud Beauville.
\newblock Vari\'{e}t\'{e}s {K}\"{a}hleriennes dont la premi\`ere classe de
  {C}hern est nulle.
\newblock {\em J. Differential Geom.}, 18(4):755--782 (1984), 1983.

\bibitem[Bon21]{bonnafe2018singular}
C\'{e}dric Bonnaf\'{e}.
\newblock Some singular curves and surfaces arising from invariants of complex
  reflection groups.
\newblock {\em Exp. Math.}, 30(3):429--440, 2021.

\bibitem[Bro10]{broue}
Michel Brou\'{e}.
\newblock {\em Introduction to complex reflection groups and their braid
  groups}, volume 1988 of {\em Lecture Notes in Mathematics}.
\newblock Springer-Verlag, Berlin, 2010.

\bibitem[BS21]{bonnafe_sarti}
C\'{e}dric Bonnaf\'{e} and Alessandra Sarti.
\newblock Complex reflection groups and {K}3 surfaces {I}.
\newblock {\em \'{E}pijournal G\'{e}om. Alg\'{e}brique}, 5:Art. 3, 26, 2021.

\bibitem[BW79]{BruceWall}
James~William Bruce and Charles Terence~Clegg Wall.
\newblock On the classification of cubic surfaces.
\newblock {\em J. London Math. Soc. (2)}, 19(2):245--256, 1979.

\bibitem[Cam16]{CAMERE_lattice_polarized}
Chiara Camere.
\newblock Lattice polarized irreducible holomorphic symplectic manifolds.
\newblock {\em Ann. Inst. Fourier (Grenoble)}, 66(2):687--709, 2016.

\bibitem[Cam18]{Camere_lattice_polarized_rmks}
Chiara Camere.
\newblock Some remarks on moduli spaces of lattice polarized holomorphic
  symplectic manifolds.
\newblock {\em Commun. Contemp. Math.}, 20(4):1750044, 29, 2018.

\bibitem[CMGHL21]{C-MGHL_Jacobians}
Sebastian Casalaina-Martin, Samuel Grushevsky, Klaus Hulek, and Radu Laza.
\newblock Complete moduli of cubic threefolds and their intermediate
  {J}acobians.
\newblock {\em Proc. Lond. Math. Soc. (3)}, 122(2):259--316, 2021.

\bibitem[CMGHL23]{C-MGHL_cohomology}
Sebastian Casalaina-Martin, Samuel Grushevsky, Klaus Hulek, and Radu Laza.
\newblock Cohomology of the {M}oduli {S}pace of {C}ubic {T}hreefolds and {I}ts
  {S}mooth {M}odels.
\newblock {\em Mem. Amer. Math. Soc.}, 282(1395), 2023.

\bibitem[CS99]{ConwaySloane}
John~Horton Conway and Neil James~Alexander Sloane.
\newblock {\em Sphere packings, lattices and groups}, volume 290 of {\em
  Grundlehren der mathematischen Wissenschaften [Fundamental Principles of
  Mathematical Sciences]}.
\newblock Springer-Verlag, New York, third edition, 1999.
\newblock With additional contributions by E. Bannai, R. E. Borcherds, J.
  Leech, S. P. Norton, A. M. Odlyzko, R. A. Parker, L. Queen and B. B. Venkov.

\bibitem[Deb22]{debarreHKbibbia}
Olivier Debarre.
\newblock Hyper-{K}\"{a}hler manifolds.
\newblock {\em Milan J. Math.}, 90(2):305--387, 2022.

\bibitem[DK07]{DolgachevKondo}
Igor~V. Dolgachev and Shigeyuki Kond{\={o}}.
\newblock {\em Moduli of K3 Surfaces and Complex Ball Quotients}, pages
  43--100.
\newblock Birkh{\"a}user Basel, Basel, 2007.

\bibitem[DR01]{DR01}
Dimitrios~I. Dais and Marko Roczen.
\newblock On the string-theoretic {E}uler numbers of 3-dimensional
  {$A$}-{$D$}-{$E$} singularities.
\newblock {\em Adv. Geom.}, 1(4):373--426, 2001.

\bibitem[Fuj87]{Fujiki19870}
Akira Fujiki.
\newblock On the de {R}ham cohomology group of a compact {K}\"{a}hler
  symplectic manifold.
\newblock In {\em Algebraic geometry, {S}endai, 1985}, volume~10 of {\em Adv.
  Stud. Pure Math.}, pages 105--165. North-Holland, Amsterdam, 1987.

\bibitem[GLS07]{GLS}
G.-M. Greuel, C.~Lossen, and E.~Shustin.
\newblock {\em Introduction to singularities and deformations}.
\newblock Springer Monographs in Mathematics. Springer, Berlin, 2007.

\bibitem[Has00]{hassett_2000}
Brendan Hassett.
\newblock Special cubic fourfolds.
\newblock {\em Compositio Math.}, 120(1):1--23, 2000.

\bibitem[Hec20]{Heckel}
Tobias Heckel.
\newblock {\em Fano schemes of lines on singular cubic hypersurfaces and their
  Picard schemes}.
\newblock PhD thesis, Hannover : Gottfried Wilhelm Leibniz Universität, Diss.,
  2020.

\bibitem[HT09]{HTmoving}
Brendan Hassett and Yuri Tschinkel.
\newblock Moving and ample cones of holomorphic symplectic fourfolds.
\newblock {\em Geom. Funct. Anal.}, 19(4):1065--1080, 2009.

\bibitem[Huy99]{Huybrechts2003CompactHM}
Daniel Huybrechts.
\newblock Compact hyper-{K}\"{a}hler manifolds: basic results.
\newblock {\em Invent. Math.}, 135(1):63--113, 1999.

\bibitem[Kon02]{Kondo}
Shigeyuki Kond{\={o}}.
\newblock The moduli space of curves of genus 4 and {D}eligne-{M}ostow's
  complex reflection groups.
\newblock In {\em Algebraic geometry 2000, {A}zumino ({H}otaka)}, volume~36 of
  {\em Adv. Stud. Pure Math.}, pages 383--400. Math. Soc. Japan, Tokyo, 2002.

\bibitem[Leh18]{lehn2015twisted}
Christian Lehn.
\newblock Twisted cubics on singular cubic fourfolds---on {S}tarr's fibration.
\newblock {\em Math. Z.}, 290(1-2):379--388, 2018.

\bibitem[LLSvS17]{LehnLehnSorgervanStraten}
Christian Lehn, Manfred Lehn, Christoph Sorger, and Duco van Straten.
\newblock Twisted cubics on cubic fourfolds.
\newblock {\em J. Reine Angew. Math.}, 731:87--128, 2017.

\bibitem[LS99]{LehrerSpringer+2011+181+194}
Gus Lehrer and Tonny~Albert Springer.
\newblock Intersection multiplicities and reflection subquotients of unitary
  reflection groups. {I}.
\newblock In {\em Geometric group theory down under ({C}anberra, 1996)}, pages
  181--193. de Gruyter, Berlin, 1999.

\bibitem[LS07]{LooijengaSwiestra}
Eduard Looijenga and Rogier Swierstra.
\newblock The period map for cubic threefolds.
\newblock {\em Compos. Math.}, 143(4):1037--1049, 2007.

\bibitem[LSV17]{LSV}
Radu Laza, Giulia Sacc\`a, and Claire Voisin.
\newblock A hyper-{K}\"ahler compactification of the intermediate {J}acobian
  fibration associated with a cubic 4-fold.
\newblock {\em Acta Math.}, 218(1):55--135, 2017.

\bibitem[Mar11]{Markman_surveyTorelli}
Eyal Markman.
\newblock A survey of {T}orelli and monodromy results for
  holomorphic-symplectic varieties.
\newblock In {\em Complex and differential geometry}, volume~8 of {\em Springer
  Proc. Math.}, pages 257--322. Springer, Heidelberg, 2011.

\bibitem[Mar13]{Markman_primeexceptional}
Eyal Markman.
\newblock Prime exceptional divisors on holomorphic symplectic varieties and
  monodromy reflections.
\newblock {\em Kyoto J. Math.}, 53(2):345--403, 2013.

\bibitem[Mon15]{Mongardi_notes}
Giovanni Mongardi.
\newblock A note on the {K}\"{a}hler and {M}ori cones of hyperk\"{a}hler
  manifolds.
\newblock {\em Asian J. Math.}, 19(4):583--591, 2015.

\bibitem[Muk03]{mukai_2003}
Shigeru Mukai.
\newblock {\em An Introduction to Invariants and Moduli}.
\newblock Cambridge Studies in Advanced Mathematics. Cambridge University
  Press, 2003.

\bibitem[Nam85]{Namikawa}
Yukihiko Namikawa.
\newblock Periods of {E}nriques surfaces.
\newblock {\em Math. Ann.}, 270(2):201--222, 1985.

\bibitem[Nik80]{Nik1980}
Viacheslav~Valentinovich Nikulin.
\newblock Integral symmetric bilinear forms and some of their applications.
\newblock 14(1):103--167, feb 1980.

\bibitem[O'G99]{OGrady10}
Kieran~G. O'Grady.
\newblock Desingularized moduli spaces of sheaves on a {$K3$}.
\newblock {\em J. Reine Angew. Math.}, 512:49--117, 1999.

\bibitem[O'G03]{OGrady6}
Kieran~G. O'Grady.
\newblock A new six-dimensional irreducible symplectic variety.
\newblock {\em J. Algebraic Geom.}, 12(3):435--505, 2003.

\bibitem[Spr74]{Springer1974}
Tonny~Albert Springer.
\newblock Regular elements of finite reflection groups.
\newblock {\em Invent. Math.}, 25:159--198, 1974.

\bibitem[Ste20]{Stegmann}
Ann-Kathrin Stegmann.
\newblock {\em Cubic fourfolds with ADE singularities and K3 surfaces}.
\newblock PhD thesis, Hannover : Gottfried Wilhelm Leibniz Universität, Diss.,
  2020.

\bibitem[Vik24]{viktorova2024classificationsingularcubicthreefolds}
Sasha Viktorova.
\newblock On the classification of singular cubic threefolds, 2024.

\bibitem[Wal99]{Wall}
Charles Terence~Clegg Wall.
\newblock Sextic curves and quartic surfaces with higher singularities.
\newblock Preprint, available at
  \url{https://www.liverpool.ac.uk/~ctcw/hsscqs.ps} – p. 7., July 19, 1999.

\bibitem[Yok02]{Yokoyama}
Mutsumi Yokoyama.
\newblock Stability of cubic 3-folds.
\newblock {\em Tokyo J. Math.}, 25(1):85--105, 2002.

\end{thebibliography}

\end{document}